\documentclass[amssymb,12pt ]{article}
\pdfoutput=1
\usepackage{geometry}
\usepackage{amsfonts}
\usepackage{amssymb}
\usepackage{amsmath}
\usepackage{amsthm}
\usepackage{graphicx}
\usepackage{picture}
\usepackage{curves}
\usepackage{subfigure}
\usepackage{epic}
\usepackage{color}

\newtheorem{thm}{Theorem}[section]

\newtheorem{lem}[thm]{Lemma}
\newtheorem{pro}[thm]{Proposition}
\newtheorem{defi}[thm]{Definition}

\newtheorem{obs}[thm]{Observation}

\newenvironment{ack}{\noindent{\bf Acknowledgments}}

\newcommand{\vol}{{\rm vol}}
\newcommand{\cs}{{\rm cs}}
\newcommand{\li}{{\rm Li}_2}

\newcommand{\modulo}{~~({\rm mod}~\pi^2)}

\begin{document}

\title{Quandle theory and the optimistic limits of the representations of link groups
}
\author{\sc Jinseok Cho}
\maketitle
\begin{abstract}
When a boudnary-parabolic representation of a link group to PSL(2,$\mathbb{C}$) is given, Inoue and Kabaya suggested a combinatorial method 
to obtain the developing map of the representation
using the octahedral triangulation and the shadow-coloring of certain quandle. Quandle is an algebraic system closely related with
the Reidemeister moves, so their method changes quite naturally under the Reidemeister moves.

In this article, we apply their method to the potential function, which was used to define the optimsitic limit,
and construct a saddle point of the function. This construction works for any boundary-parabolic representation,
and it shows that the octahedral triangulation is good enough to study all possible boundary-parabolic representations of the link group.
Furthermore the evaluation of the potential function at the saddle point
becomes the complex volume of the representation, and this saddle point changes naturally under the Reidemeister moves
because it is constructed using the quandle.
\end{abstract}

\section{Introduction}\label{sec1}

A link $L$ has the hyperbolic structure when there exists a discrete faithful representation
$\rho:\pi_1(L)\rightarrow{\rm PSL}(2,\mathbb{C})$, where {\it the link group} $\pi_1(L)$ is the fundamental group of the link complement $\mathbb{S}^3\backslash L$.
The standard method to find the hyperbolic structure of $L$ is to consider some triangulation of $\mathbb{S}^3\backslash L$ and
solve certain set of equations. (These equations are called {\it the hyperbolicity equations}.) 
Each solution determines a boundary-parabolic representation\footnote{
{\it Boundary-parabolic} means the image of the peripheral subgroup $\pi_1(\partial(\mathbb{S}^3\backslash L))$ is a parabolic subgroup of PSL(2, $\mathbb{C}$). Note that the geometric representation is boundary-parabolic.}
 and one of them is {\it the geometric representation}, which means
the determined boundary-parabolic representation is discrete and faithful.
Due to Mostow's rigidity theorem, the hyperbolic structure of a link is a topological property.
Therefore, {it is natural to expect the invariance of the hyperbolic structure under the Reidemeister moves.}
However, this could not be seen easily because, even small {change} on the triangulation changes the solution radically.

Recently, Inoue and Kabaya, in \cite{Kabaya14}, developed a method to construct the hyperbolic structure of $L$
using the link diagram and the geometric representation. More generally, when a boundary-parabolic representation $\rho$
is given, they constructed the explicit geometric shapes of the tetrahedra of certain triangulation
using $\rho$. Their main method is to construct the geometric shapes using certain quandle homology,
which is defined directly from the link diagram $D$ and the representation $\rho$.
Here, quandle is an algebric system whose axioms are closely related with the Reidemeister moves of link diagrams,
so their construction changes quite naturally under the Reidemeister moves.
(The definition of the quandle is in Section \ref{sec21}. A good survey of quandle is the book \cite{quandle}.)
{ The result \cite{Kabaya14} suggests} a combinatorial method to obtain 
the hyperbolic structure of the link complement.

Interestingly, the triangulation they used in \cite{Kabaya14} was also used to define the optimistic limit
of the Kashaev invariant {in} \cite{Cho13a}. 
As a matter of fact, this triangulation arises naturally from the link diagram.
(See Section 3 of \cite{Weeks05} and Section \ref{triang} of this article for the definition.)
We call this triangulation {\it octahedral triangulation} of $\mathbb{S}^3\backslash(L\cup\{\text{two points}\})$
associated with the link diagram $D$.

The optimistic limit first was appeared in \cite{Kashaev95} when Kashaev proposed the volume conjecture.
This conjecture relates certain limits of link invariants, called Kashaev invariants, with the hyperbolic volumes. 
The optimistic limit{, which was first defined in \cite{Murakami00a},} is the value of certain potential function evaluated at a saddle point,
where the function and the value are expected to be an analytic continuation of the Kashaev invariant and
the limit of the invariant, respectively. 
As a matter of fact, physicists usually call the evaluation {\it the classical limit} 
and consider it {\it the actual limit} of the invariant. 
{ Mathematically rigorous definition of the optimistic limit was proposed in \cite{Yokota10}
and the value was proved to coincide with the hyperbolic volume.}
The author and several others developed several versions of the optimistic limit
in many articles, but we will modify the version of \cite{Cho13a} in this article {so as to construct a solution without solving equations}. 


The optimistic limit is defined by the potential function $V(z_1,\ldots,z_n,w_k^j,\ldots)$. 
Previously, in \cite{Cho13a}, this function was defined purely by the link diagram, but here we modify it using the information of
the representation $\rho$. (The definition is in Section \ref{sec3}.) We consider a solution of the following set
\begin{equation*}
\mathcal{H}:=\left\{\left.\exp(z_k \frac{\partial V}{\partial z_k})=1,~
  \exp(w_k^j\frac{\partial V}{\partial w_k^j})=1\right|j:\text{degenerate  crossings}, ~k=1,\ldots,n\right\},
\end{equation*}
which is a saddle-point of the potential function $V$. Then {Proposition \ref{pro1}} will show that $\mathcal{H}$
becomes the hyperbolicity equations of the octahedral triangulation. 

{Solving the equations in $\mathcal{H}$ is not easy because there are infinitely many solutions.
The standard way to avoid this difficulty is to deform the octahedral triangulation of $\mathbb{S}^3\backslash(L\cup\{\text{two points}\})$ 
to the triangulation of $\mathbb{S}^3\backslash L$, as in \cite{Yokota10}.
However, this deformation produces the problem of the existence of solutions because some triangulation constructed from a link diagram 
may have no solution. (A recent paper \cite{Yokota16} proved the existence of solutions for the alternating links.)
Furthermore, the author believes these deformation of the triangulation loses the combinatorial properties of link diagrams.
Therefore, we will use the octahedral triangulation without any deformation and do not solve the equations in $\mathcal{H}$.
Instead,}
we will construct an explicit solution $(z_1^{(0)},\ldots,z_n^{(0)},(w_k^j)^{(0)},\ldots)$ of $\mathcal{H}$.
{
\begin{thm} There exists a formula to construct a solution $(z_1^{(0)},\ldots,z_n^{(0)},(w_k^j)^{(0)},\ldots)$ of $\mathcal{H}$
by using the quandle associated with the representation $\rho$. (The exact formulas are in {Theorem \ref{thm1}.})
\end{thm}}

{ The evaluation of the potential function $V$ depends on the the choice of log-branch. 
To obtain a well-defined value, modify the potential function to}
\begin{eqnarray*}
\lefteqn{V_0(z_1,\ldots,z_n,(w_k^j),\ldots):=V(z_1,\ldots,z_n,(w_k^j),\ldots)}\label{V0} \\
&&-\sum_k\left(z_k \frac{\partial V}{\partial z_k}\right)\log z_k
-\sum_{ j, k}\left(w_k^j \frac{\partial V}{\partial w_k^j}\right)\log w_k^j.
\end{eqnarray*}

{
\begin{thm} For the constructed solution $(z_1^{(0)},\ldots,z_n^{(0)},(w_k^j)^{(0)},\ldots)$ of $\mathcal{H}$
and the modified potential function $V_0$ above, the following holds:
\begin{equation}\label{optimistic}
  V_0(z_1^{(0)},\ldots,z_n^{(0)},(w_k^j)^{(0)},\ldots)\equiv i(\vol(\rho)+i\,\cs(\rho))~~({\rm mod}~\pi^2),
\end{equation}
where $\vol(\rho)$ and $\cs(\rho)$ are the hyperbolic volume and the Chern-Simons invariant of $\rho$
defined in \cite{Zickert09}, respectively. 
\end{thm}}

The proof will be in Theorem \ref{thm2}.
The left-hand side of (\ref{optimistic}) is called {\it the optimistic limit of $\rho$},
and $\vol(\rho)+i\,\cs(\rho)$ in the right-hand side is called {\it the complex volume of $\rho$}.

Note that for any boundary-parabolic representation $\rho$, we can always construct the solution associated with $\rho$.
This implies that the octahedral triangulation is good enough for the study of all possible boundary-parabolic representations
from the link group to ${\rm PSL}(2,\mathbb{C})$. The set of all possible representations can be regarded as
{\it the Ptolemy variety} (see \cite{Zickert15b} for detail) and we expect the octahedral triangulation will be
very useful to the study of the Ptolemy variety. (Actual application to the Ptolemy variety is in preparation now.)

Furthermore, the construction of the solution is based on the quandle in \cite{Kabaya14}.
Therefore, this solution changes locally under the Reidemeister moves.
This implies that we can explore the hyperbolic structure of a link by finding the solution
and keeping track of the changes of the solution under the Reidemeister moves.
As a matter of fact, after the appearance of the first draft of this article, this idea was successfully
{ in \cite{Cho14c}, \cite{Cho15a} and more applications} are in preparation.

Among the applications, we remark that the article \cite{Cho14c} contains very similar results with this article.
Both articles construct the solution associated with $\rho$ using the same quandle.
However, the major differences are the triangulations.
Both uses the same {\it octahedral decomposition} of $\mathbb{S}^3\backslash(L\cup\{\text{two points}\})$,
but this article uses the subdivision of each octahedron into four tetrahedra and call the result 
{\it four-term (or octahedral) triangulation}, 
whereas the article \cite{Cho14c} uses the subdivision of the same octahedron into five tetrahedra and call the result {\it five-term triangulation}.
Some {tetrahedra} in the four-term triangulation can be degenerate and this introduces technical difficulties.
However, the five-term triangulation used in \cite{Cho14c} does not contain any degenerate {tetrahedra},
so it is far easier and convenient. (That is why this article is three times longer than \cite{Cho14c}.)
As a conclusion, this article contains the original idea of using quandle to construct the solution
and the article \cite{Cho14c} improved the idea.

This article consists of the following contents. In Section \ref{sec2}, we will summarize some results of \cite{Kabaya14}.
Especially, the definition of the quandle and the octahedral triangulation will appear.
Section \ref{sec3} will define the optimistic limit and the hyperbolicity equations.
The main formula (Theorem \ref{thm2}) of the solution associated with the given representation $\rho$ will appear.
Section \ref{sec4} will discuss two simple examples, the figure-eight knot $4_1$ and the trefoil knot $3_1$.

\section{Quandle}\label{sec2}

In this section, we will survey some results of the article \cite{Kabaya14}. 
We remark that all formulas of this section come from \cite{Kabaya14} and the author learned them
from the series lectures of Ayumu Inoue given at
Seoul National University during spring of 2012.

\subsection{Conjugation quandle of parabolic elements}\label{sec21}

\begin{defi}A \textbf{quandle} is a set $X$ with a binary operation $*$ satisfying the following three conditions:
\begin{enumerate}
  \item $a*a=a$ for any $a\in X$,
  \item the map $*b:X\rightarrow X \;(a\mapsto  a*b)$ is bijective for any $b\in X$,
  \item  $(a*b)*c=(a*c)*(b*c)$ for any $a,b,c\in X$.
\end{enumerate}
\end{defi}

The inverse of $*b$ is notated by $*^{-1}b$. In other words, the equation $a*^{-1}b=c$ is equivalent to $c*b=a$.

\begin{defi} Let $G$ be a group and $X$ be a subset of $G$ satisfying 
  \begin{equation*}
    g^{-1}Xg=X \text{ for any }g\in G.
  \end{equation*}
Define the binary operation $*$ on $X$ by
\begin{equation}\label{eq1}
  a*b=b^{-1} a b
\end{equation}
for any $a,b\in X$. Then $(X,*)$ becomes a quandle and is called the \textbf{conjugation quandle}.
\end{defi}

As an example, let $\mathcal{P}$ be the set of parabolic elements of ${\rm PSL}(2,\mathbb{C})={\rm Isom^+}(\mathbb{H}^3)$. 
Then 
  $$g^{-1}\mathcal{P}g=\mathcal{P}$$
holds for any $g\in {\rm PSL}(2,\mathbb{C})$. Therefore, $(\mathcal{P},*)$ is a conjugation quandle,
and this is the only quandle we are using in this article.

To perform concrete calculations, explicit expression of $(\mathcal{P},*)$ was introduced {in} \cite{Kabaya14}.
At first, note that 

$$\left(\begin{array}{cc}p & q \\r & s\end{array}\right)^{-1}
  \left(\begin{array}{cc}1 & 1 \\0 & 1\end{array}\right)
  \left(\begin{array}{cc}p & q \\r & s\end{array}\right)
  =\left(\begin{array}{cc}1+rs & s^2 \\-r^2 & 1-rs\end{array}\right),$$
for $\left(\begin{array}{cc}p & q \\r & s\end{array}\right)\in{\rm PSL}(2,\mathbb{C})$.
Therefore, we can identify $(\mathbb{C}^2\backslash\{0\})/\pm$ with $\mathcal{P}$  by
\begin{equation}\label{eq2}
  \left(\begin{array}{cc}\alpha & \beta\end{array}\right)
  \longleftrightarrow\left(\begin{array}{cc}1+\alpha\beta & \beta^2 \\-\alpha^2 & 1-\alpha\beta\end{array}\right),
\end{equation}
where $\pm$ means the equivalence relation 
  $\left(\begin{array}{cc}\alpha & \beta\end{array}\right)\sim\left(\begin{array}{cc}-\alpha & -\beta\end{array}\right)$.
We define the operation $*$ on $\mathcal{P}$ by
\begin{eqnarray*}
  \left(\begin{array}{cc}\alpha & \beta\end{array}\right)*  \left(\begin{array}{cc}\gamma & \delta\end{array}\right)
  :=\left(\begin{array}{cc}\alpha & \beta\end{array}\right)
  \left(\begin{array}{cc}1+\gamma\delta &\delta^2 \\   -\gamma^2& 1-\gamma\delta\end{array}\right)
  \in(\mathbb{C}^2\backslash\{0\})/\pm,
\end{eqnarray*}
where the matrix multiplication on the right-hand side is the standard multiplication. 
(This definition is the transpose of the one used in \cite{Kabaya14} and \cite{Cho14c}.)
Note that this definition coincides with the operation of the conjugation quandle $(\mathcal{P},*)$ by
\begin{eqnarray*}
  &&\left(\begin{array}{cc}\alpha & \beta\end{array}\right)*  \left(\begin{array}{cc}\gamma & \delta\end{array}\right)
  = \left(\begin{array}{cc}\alpha & \beta\end{array}\right)
  \left(\begin{array}{cc}1+\gamma\delta &\delta^2 \\   -\gamma^2& 1-\gamma\delta\end{array}\right)
  \in(\mathbb{C}^2\backslash\{0\})/\pm\\
  &&\longleftrightarrow  \left(\begin{array}{cc}1+\gamma\delta &\delta^2 \\   -\gamma^2& 1-\gamma\delta\end{array}\right)^{-1}
  \left(\begin{array}{cc}1+\alpha\beta & -\alpha^2 \\ \beta^2& 1-\alpha\beta\end{array}\right)
   \left(\begin{array}{cc}1+\gamma\delta &\delta^2 \\   -\gamma^2& 1-\gamma\delta\end{array}\right)\\
  &&~~~~~= \left(\begin{array}{cc}\gamma & \delta\end{array}\right)^{-1}\left(\begin{array}{cc}\alpha & \beta\end{array}\right)
     \left(\begin{array}{cc}\gamma & \delta\end{array}\right)\in{\rm PSL}(2,\mathbb{C}).
\end{eqnarray*}
The inverse operation is given by
$$\left(\begin{array}{cc}\alpha & \beta\end{array}\right)*^{-1}  \left(\begin{array}{cc}\gamma & \delta\end{array}\right)
  = \left(\begin{array}{cc}\alpha & \beta\end{array}\right)
  \left(\begin{array}{cc}1-\gamma\delta & -\gamma^2 \\ \delta^2 & 1+\gamma\delta\end{array}\right).$$
From now on, we use the notation $\mathcal{P}$ instead of $(\mathbb{C}^2\backslash\{0\})/\pm$.

\subsection{Link group and shadow-coloring}\label{sec22}

Consider a representation $\rho:\pi_1(L)\rightarrow {\rm PSL}(2,\mathbb{C})$ of a hyperbolic link $L$.
We call $\rho$ {\it boundary-parabolic} when the peripheral subgroup $\pi_1(\partial (\mathbb{S}^3\backslash L))$ of $\pi_1(L)$
maps to a subgroup of ${\rm PSL}(2,\mathbb{C})$ whose elements are all parabolic.

 For a fixed oriented link diagram\footnote{
We always assume the diagram does not contain a trivial knot component which has only over-crossings or under-crossings
or no crossing. (For example, any unseparable link diagram satisfies this condition.)
If it happens, then we change the diagram of the trivial component slightly.
For example, applying Reidemeister second move to make different types of crossings 
or Reidemeister first move to add a kink is good enough. 
This assumption is necessary to guarantee that
the octahedral triangulation becomes a topological triangulation of $\mathbb{S}^3\backslash(L\cup\{\text{two points}\})$} $D$ of $L$, Wirtinger presentation gives an algorithmic expression of $\pi_1(L)$.
For each arc $\alpha_k$ of $D$, we draw a small arrow labelled $a_k$ as in Figure \ref{pic01}, which presents a loop. 
(The details are in \cite{Rolfsen}. Here we are using the opposite orientation of $a_k$
to be consistent with the operation of the conjugation quandle.)
This loop corresponds to one of the meridian curves of the boundary tori, so $\rho(a_k)$ is
an element in $\mathcal{P}$. Hence we call $\{\rho(a_1),\ldots,\rho(a_n)\}$ {\it arc-coloring}\footnote{
Strictly speaking, arc-coloring is a map from arcs of $D$ to $\mathcal{P}$, not a set. 
(Region-coloring, which will be defined below, is also a map from regions of $D$ to $\mathcal{P}$.) However, 
we abuse the set notation here for convenience.} of $D$,
whereas each $\rho(a_k)$ is assigned to the corresponding arc $\alpha_k$.

\begin{figure}[h]
\centering
  \includegraphics[scale=0.45]{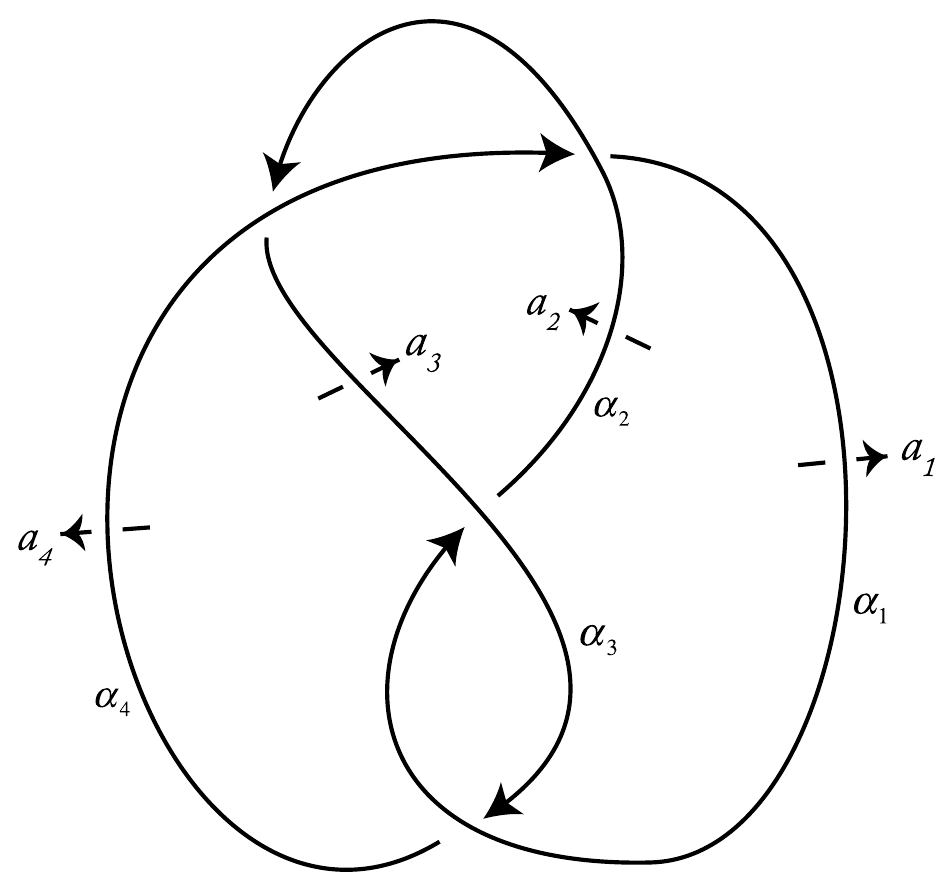}
  \caption{The figure-eight knot $4_1$}\label{pic01}
\end{figure}

Wirtinger presentation of the link group is given by
$$\pi_1(L)=<a_1,\ldots,a_n;r_1,\ldots,r_n>,$$
where the relation $r_l$ is assigned to each crossing as in Figure \ref{pic02}. Note that $r_l$ coincides with (\ref{eq1}),
so we can write down relation of the arc-colors as in Figure \ref{pic03}.

\begin{figure}[h]
\centering
\subfigure[$r_l : a_{l+1}=a_k^{-1} a_l a_k$]{
  \includegraphics[scale=0.4]{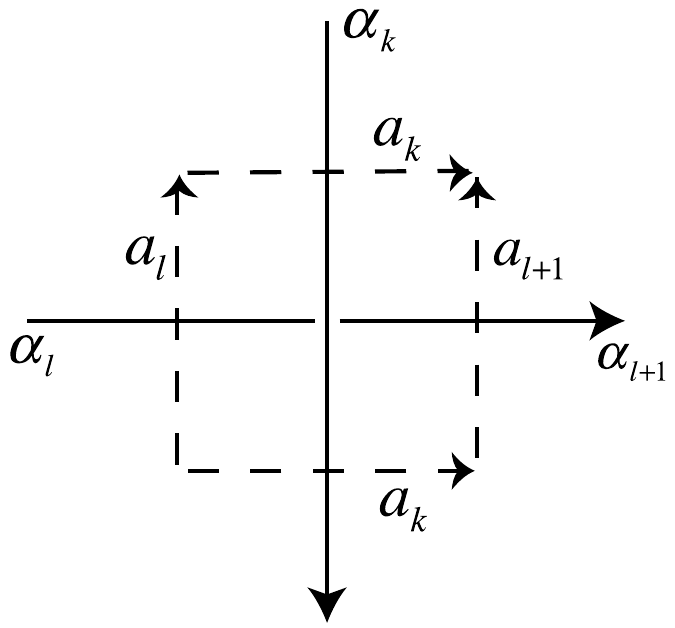}}\hspace{1cm}
  \subfigure[$r_l : a_l=a_k^{-1} a_{l+1}a_k$]{
  \includegraphics[scale=0.4]{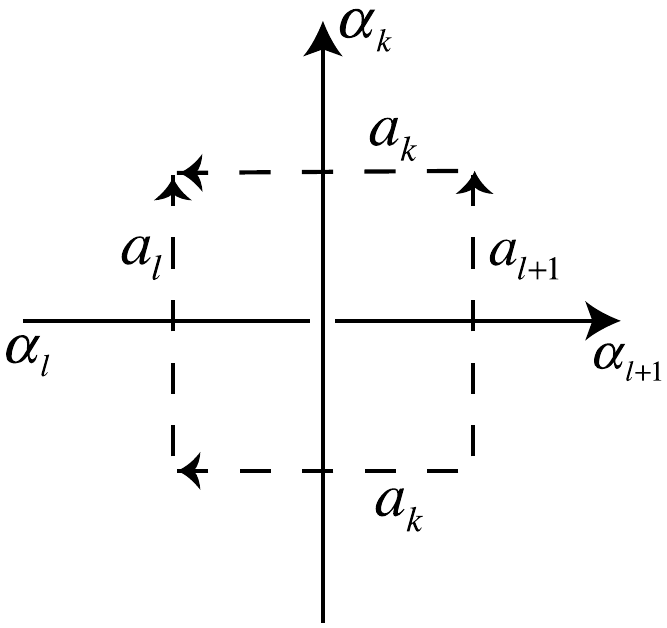} } 
  \caption{Relations at crossings}\label{pic02}
\end{figure}

\begin{figure}[h]
\centering  \setlength{\unitlength}{0.6cm}\thicklines
\begin{picture}(6,6)  
    \put(6,5){\vector(-1,-1){4}}
    \put(2,5){\line(1,-1){1.8}}
    \put(4.2,2.8){\line(1,-1){1.8}}
    \put(6.2,5.2){$\rho(a_k)$}
    \put(1.3,5.2){$\rho(a_l)$}
    \put(4.5,0.3){$\rho(a_l)*\rho(a_k)$}
  \end{picture}
  \caption{Arc-coloring}\label{pic03}
\end{figure}
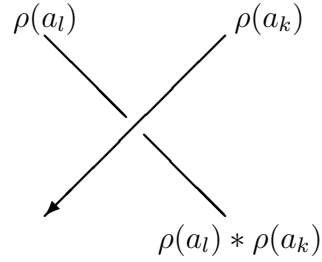

From now on, we always assume $\rho:\pi_1(L)\rightarrow {\rm PSL}(2,\mathbb{C})$
 is a given boundary-parabolic representation. 
To avoid redundant notations, {arc-coloring will be} denoted by $\{a_1,\ldots,a_n\}$ without { indicating} $\rho$ from now on.
Choose an element {$s_f\in\mathcal{P}$} corresponding to {a} region of the diagram $D$ and determine
{$s_1,s_2,\ldots,s_m\in\mathcal{P}$ corresponding to each regions} using the relation in Figure \ref{pic04}.

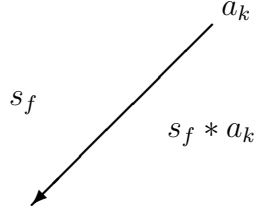
\begin{figure}[h]
\centering  \setlength{\unitlength}{0.6cm}\thicklines
\begin{picture}(6,5)  
    \put(6,4){\vector(-1,-1){4}}
    \put(1.5,2.2){$s_f$}
    \put(5,1.5){$s_f*a_k$}
    \put(6.2,4.2){$a_k$}
  \end{picture}
  \caption{Region-coloring}\label{pic04}
\end{figure}

The assignment of elements of $\mathcal{P}$ to all regions using the relation in Figure \ref{pic04} is called {\it the region-coloring}. This assignment is well-defined
because the two curves in Figure \ref{pic05}, which we call {\it the cross-changing pair}, determine the same region-coloring, and 
any pair of curves with the same starting and ending points can be transformed each other by finite sequence of cross-changing pairs.

\begin{figure}[h]
\centering  \setlength{\unitlength}{0.9cm}\thicklines
\subfigure[Positive crossing]{
  \includegraphics[scale=0.8]{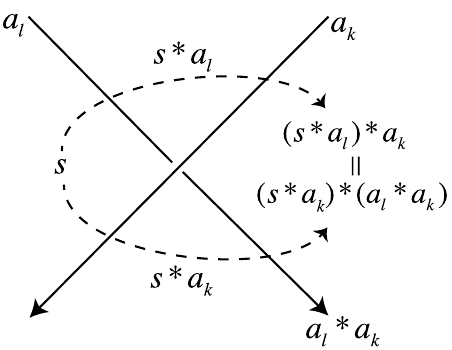}}  \hspace{1cm}
\subfigure[Negative crossing]{
  \includegraphics[scale=0.8]{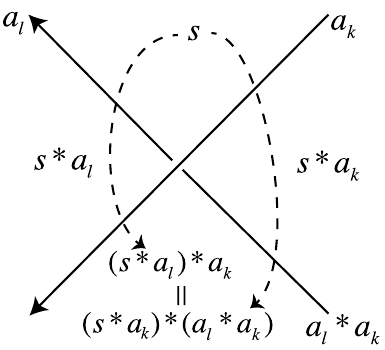}}  
  \caption{Well-definedness of region-coloring}\label{pic05}
\end{figure}

An arc-coloring together with a region-coloring is called {\it an shadow-coloring}.
The following lemma shows important property of shadow-colorings, which is crucial for showing the existence of solutions of certain equations.

\begin{defi}\label{Hopf} 
The \textbf{Hopf map} $h:\mathcal{P}\longrightarrow \mathbb{CP}^1=\mathbb{C}\cup\{\infty\}$ is defined by
$$\left(\begin{array}{cc}\alpha &\beta\end{array}\right)\mapsto \frac{\alpha}{\beta}.$$
Note that $h\left(\begin{array}{cc}\alpha &\beta\end{array}\right)=\frac{\alpha}{\beta}$ 
is the fixed point of the M\"{o}bius transformation $f(z)=\frac{(1+\alpha\beta)z-\alpha^2}{\beta^2 z+(1-\alpha\beta)}$.
\end{defi}

\begin{lem}\label{lem1} Let $L$ be a link and assume an arc-coloring is already given 
by the boundary-parabolic representation $\rho:\pi_1(L)\longrightarrow {\rm PSL}(2,\mathbb{C})$.
{Then there exists a region-coloring such that, 
for any edge of the link diagram with its arc-color $a_k$ ($k=1,\ldots,n$) and its surrounding region-colors {$s_f, s_f*a_k$} 
(see Figure \ref{pic04}), the following holds:}
\begin{equation}\label{regcon}
  h(a_k)\neq h({s_f})\neq h({s_f}*a_k)\neq h(a_k)
\end{equation}
{holds}.
\end{lem}




\begin{proof}
Note that this was already proved inside the proof of Proposition 2 of \cite{Kabaya14}.
However, finding out the proof in the article is not easy, so we write it down below for the readers' convenience.

For the given arc-colors $a_1,\ldots,a_n$, we choose region-colors $s_1,\ldots,s_m$ so that
\begin{equation}\label{exi}
  \{h(s_1),\ldots,h(s_m)\}\cap\{h(a_1),\ldots,h(a_n)\}=\emptyset.
\end{equation}
This is always possible because, 
{
each $h(s_k)$ is written as $h(s_k)=M_k(h(s_1))$ by a M\"{o}bius transformation $M_k$, which only depends on
the arc-colors $a_1,\dots,a_r$. If we choose $h(s_1)\in\mathbb{CP}^1$ away from the finite set
$$\bigcup_{1\leq k\leq n}\left\{M_k^{-1}(h(a_1)),\ldots,M_k^{-1}(h(a_r))\right\},$$
we have $h(s_k)\notin \{h(a_1),\ldots,h(a_r)\}$ for all $k$.
This choice of a region-coloring guarantees $ h(a_k)\neq h({s_f})$ and $h({s_f}*a_k)\neq h(a_k)$.

Now assume $h(s_f*a_k)=h(s_f)$ holds under the choice of the region-coloring above. Then we obtain
\begin{equation}\label{eqnh}
  h(s_f*a_k)=\widehat{a_k}(h(s_f))=h(s_f),
\end{equation}
where $\widehat{a_k}:\mathbb{CP}^1\rightarrow\mathbb{CP}^1$ is the M\"{o}bius transformation
$$\widehat{a_k}(z)=\frac{(1+\alpha_k\beta_k)z-\alpha_k^2}{\beta_k^2 z+(1-\alpha_k\beta_k)}$$
of $a_k=\left(\begin{array}{cc}\alpha_k &\beta_k\end{array}\right)$. Then (\ref{eqnh}) implies
$h(s)$ is the fixed point of $\widehat{a_k}$, which means $h(a_k)=h(s)$ that contradicts (\ref{exi}).}

\end{proof}

We remark that the condition (\ref{exi}) {of a region-coloring} is stronger than the condition in { Lemma \ref{lem1}.
For example, the region-colorings of the examples in Section \ref{sec4} satisfy Lemma \ref{lem1}, but they do not satisfy (\ref{exi}).
Even though we actually proved stronger condition (\ref{exi}) in the proof, 
the region-colorings we consider are always assumed to satisfy Lemma \ref{lem1} from now on.}
The arc-coloring induced by $\rho$ together with the region-coloring satisfying Lemma \ref{lem1} is called 
the {\it shadow-coloring induced by} $\rho$.
This shadow-coloring will determine the exact coordinates of points of the octahedral triangulation in the next section.

\subsection{Octahedral triangulations of link complements}\label{triang}

In this section, we describe the ideal triangulation of $\mathbb{S}^3\backslash (L\cup\{\text{two points}\})$ which appeared in \cite{Cho13a}.
{Note that this triangulation naturally arises from the link diagram and has been widely used in various names. 
For example, the famous software SnapPea used this triangulation to obtain 
an ideal triangulation of the link complement $\mathbb{S}^3\backslash L$ \cite{Weeks05} (see also \cite{Yokota10}.) 
Another name of this construction is the {\it tunnel construction} in \cite{Baseilhac07}.
It seems the first written appearance of this construction was in \cite{Thurston99}.}

To obtain the triangulation, we consider the crossing $j$ in Figure \ref{crossings} 
and place an octahedron ${\rm A}_j{\rm B}_j{\rm C}_j{\rm D}_j{\rm E}_j{\rm F}_j$ 
on each crossing $j$ as in Figure \ref{twistocta}(a). Then we twist the octahedron by identifying edges ${\rm B}_j{\rm F}_j$ to ${\rm D}_j{\rm F}_j$ and
${\rm A}_j{\rm E}_j$ to ${\rm C}_j{\rm E}_j$, respectively. The edges ${\rm A}_j{\rm B}_j$, ${\rm B}_j{\rm C}_j$,
${\rm C}_j{\rm D}_j$ and ${\rm D}_j{\rm A}_j$ are called {\it horizontal edges} and we sometimes express these edges
in the diagram as arcs around the crossing as in Figure \ref{crossings}.

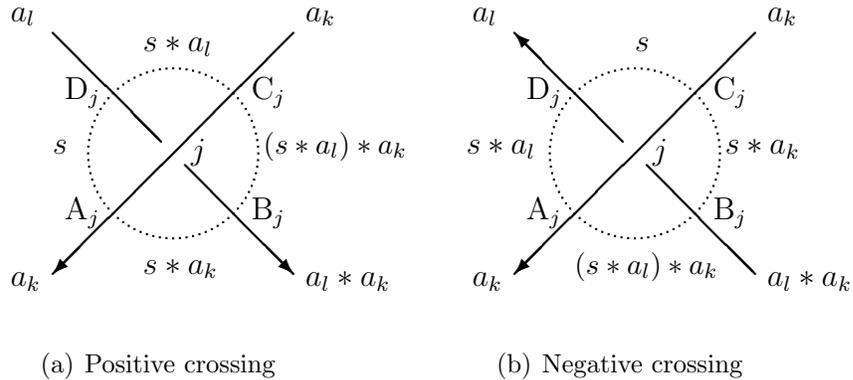
\begin{figure}[h]
\centering
\subfigure[Positive crossing]{\begin{picture}(6,5.5)  
  \setlength{\unitlength}{0.8cm}\thicklines
        \put(4,3){\arc[5](1,1){360}}
    \put(6,5){\vector(-1,-1){4}}
    \put(2,5){\line(1,-1){1.8}}
    \put(4.2,2.8){\vector(1,-1){1.8}}
    \put(2.2,1.9){${\rm A}_j$}
    \put(5.3,1.9){${\rm B}_j$}
    \put(5.3,3.9){${\rm C}_j$}
    \put(2.2,3.9){${\rm D}_j$}
    \put(4.3,2.9){$j$}
    \put(6.2,5.2){$a_k$}
    \put(1.3,0.8){$a_k$}
    \put(1.3,5.2){$a_l$}
    \put(6.2,0.8){$a_l*a_k$}
    \put(3.5,4.7){$s*a_l$}
    \put(2,3){$s$}
    \put(5.5,3){\small $(s*a_l)*a_k$}
    \put(3.5,1){$s*a_k$}
  \end{picture}}
\subfigure[Negative crossing]{\begin{picture}(6,5.5)  
  \setlength{\unitlength}{0.8cm}\thicklines
        \put(4,3){\arc[5](1,1){360}}
    \put(6,5){\vector(-1,-1){4}}
    \put(3.8,3.2){\vector(-1,1){1.8}}
    \put(4.2,2.8){\line(1,-1){1.8}}
    \put(2.2,1.9){${\rm A}_j$}
    \put(5.3,1.9){${\rm B}_j$}
    \put(5.3,3.9){${\rm C}_j$}
    \put(2.2,3.9){${\rm D}_j$}
    \put(4.3,2.9){$j$}
    \put(6.2,5.2){$a_k$}
    \put(1.3,0.8){$a_k$}
    \put(1.3,5.2){$a_l$}
    \put(6.2,0.8){$a_l*a_k$}
    \put(4,4.7){$s$}
    \put(1.2,3){$s*a_l$}
    \put(5.5,3){$s*a_k$}
    \put(3,1){\small $(s*a_l)*a_k$}
  \end{picture}}
 \caption{Crossing $j$ with shadow-coloring}\label{crossings}
\end{figure}

\begin{figure}[h]
\centering
\subfigure[]{\includegraphics[scale=0.7]{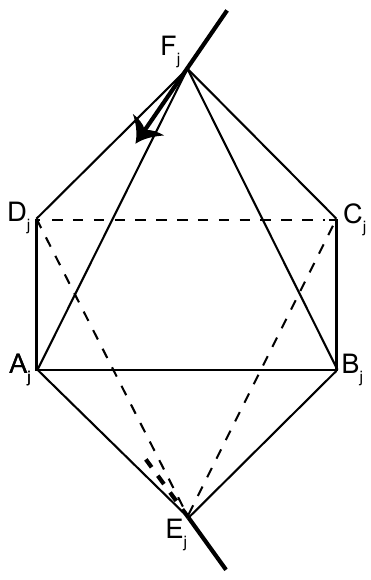}}\hspace{0.5cm}
\subfigure[]{\includegraphics[scale=0.7]{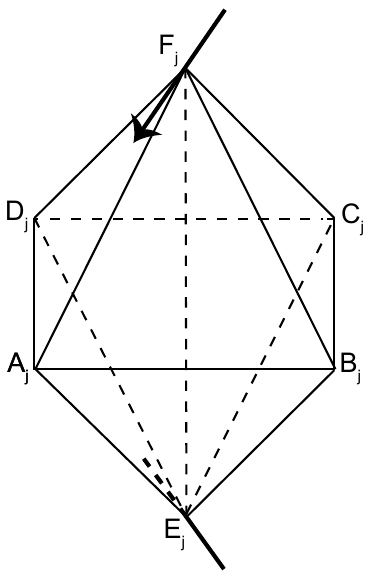}}\hspace{0.5cm}
\subfigure[]{\includegraphics[scale=0.7]{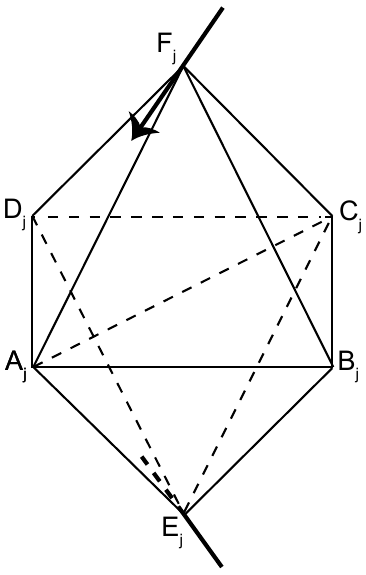}}
 \caption{Octahedron on the crossing $j$}\label{twistocta}
\end{figure}

Then we glue faces of the octahedra following the lines of the link diagram. 
Specifically, there are three gluing patterns as in Figure \ref{glue pattern}.
In each cases (a), (b) and (c), we identify the faces
$\triangle{\rm A}_{j}{\rm B}_{j}{\rm E}_{j}\cup\triangle{\rm C}_{j}{\rm B}_{j}{\rm E}_{j}$ to
$\triangle{\rm C}_{j+1}{\rm D}_{j+1}{\rm F}_{j+1}\cup\triangle{\rm C}_{j+1}{\rm B}_{j+1}{\rm F}_{j+1}$,
$\triangle{\rm B}_{j}{\rm C}_{j}{\rm F}_{j}\cup\triangle{\rm D}_{j}{\rm C}_{j}{\rm F}_{j}$ to
$\triangle{\rm D}_{j+1}{\rm C}_{j+1}{\rm F}_{j+1}\cup\triangle{\rm B}_{j+1}{\rm C}_{j+1}{\rm F}_{j+1}$
and
$\triangle{\rm A}_{j}{\rm B}_{j}{\rm E}_{j}\cup\triangle{\rm C}_{j}{\rm B}_{j}{\rm E}_{j}$ to
$\triangle{\rm C}_{j+1}{\rm B}_{j+1}{\rm E}_{j+1}\cup\triangle{\rm A}_{j+1}{\rm B}_{j+1}{\rm E}_{j+1}$,
respectively.

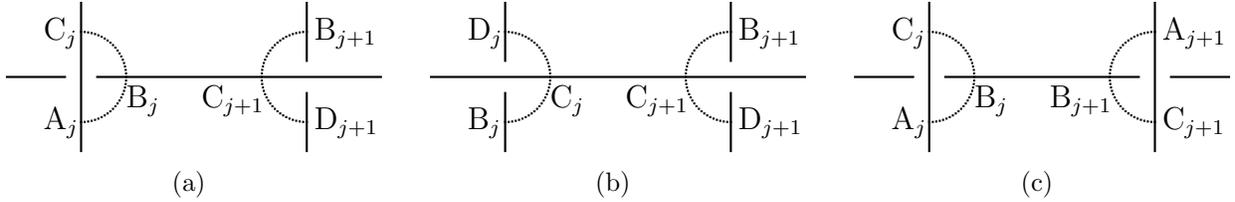
\begin{figure}[h]
\centering
  \subfigure[]
  {\begin{picture}(5,2)\thicklines
   \put(1,1){\arc[5](0,-0.6){180}}
   \put(4,1){\arc[5](0,0.6){180}}
   \put(1.2,1){\line(1,0){3.8}}
   \put(1,2){\line(0,-1){2}}
   \put(4,0){\line(0,1){0.8}}
   \put(4,2){\line(0,-1){0.8}}
   \put(0.8,1){\line(-1,0){0.8}}
   \put(0.5,0.3){${\rm A}_j$}
   \put(1.6,0.6){${\rm B}_j$}
   \put(0.5,1.5){${\rm C}_j$}
   \put(4.1,0.3){${\rm D}_{j+1}$}
   \put(2.6,0.6){${\rm C}_{j+1}$}
   \put(4.1,1.5){${\rm B}_{j+1}$}
  \end{picture}}\hspace{0.5cm}
  \subfigure[]
  {\begin{picture}(5,2)\thicklines
   \put(1,1){\arc[5](0,-0.6){180}}
   \put(4,1){\arc[5](0,0.6){180}}
   \put(5,1){\line(-1,0){5}}
   \put(1,2){\line(0,-1){0.8}}
   \put(1,0){\line(0,1){0.8}}
   \put(4,2){\line(0,-1){0.8}}
   \put(4,0){\line(0,1){0.8}}
   \put(0.5,0.3){${\rm B}_j$}
   \put(1.6,0.6){${\rm C}_j$}
   \put(0.5,1.5){${\rm D}_j$}
   \put(4.1,0.3){${\rm D}_{j+1}$}
   \put(2.6,0.6){${\rm C}_{j+1}$}
   \put(4.1,1.5){${\rm B}_{j+1}$}
  \end{picture}}\hspace{0.5cm}
  \subfigure[]
  {\begin{picture}(5,2)\thicklines
   \put(1,1){\arc[5](0,-0.6){180}}
   \put(4,1){\arc[5](0,0.6){180}}
   \put(4.2,1){\line(1,0){0.8}}
   \put(1,2){\line(0,-1){2}}
   \put(0.8,1){\line(-1,0){0.8}}
   \put(4,2){\line(0,-1){2}}
   \put(1.2,1){\line(1,0){2.6}}
   \put(0.5,0.3){${\rm A}_j$}
   \put(1.6,0.6){${\rm B}_j$}
   \put(0.5,1.5){${\rm C}_j$}
   \put(4.1,0.3){${\rm C}_{j+1}$}
   \put(2.6,0.6){${\rm B}_{j+1}$}
   \put(4.1,1.5){${\rm A}_{j+1}$}
  \end{picture}}
  \caption{Three gluing patterns}\label{glue pattern}
\end{figure} 

Note that this gluing process identifies vertices $\{{\rm A}_j, {\rm C}_j\}$ to one point, denoted by $-\infty$,
and $\{{\rm B}_j, {\rm D}_j\}$ to another point, denoted by $\infty$, and finally $\{{\rm E}_j, {\rm F}_j\}$ to
the other points, denoted by ${\rm P}_t$ where $t=1,\ldots,c$ and $c$ is the number of the components of the link $L$. 
The regular neighborhoods of $-\infty$ and $\infty$ are two 3-balls and that of $\cup_{t=1}^c P_t$ is
a tubular neighborhood of the link $L$. 
Therefore, after removing all vertices of the gluing, 
we obtain an {\it octahedral decomposition} of $\mathbb{S}^3\backslash (L\cup\{\pm\infty\})$.
The {\it octahedral triangulation} is obtained by subdividing each octahedron of the decomposition 
into four tetrahedra in certain way.

{To apply the construction of the developing map of $\rho$ in Theorem 4.11 of \cite{Zickert09}, 
we subdivide each octahedron into four tetrahedra using the shadow-coloring of $\rho$ as follows.
}

{
\begin{defi}
Consider a crossing $j$ with the shadow-color in Figure \ref{crossings}. 
The crossing $j$ is called \textbf{non-degenerate} when $h(a_k)\neq h(a_l)$
and \textbf{degenerate} when $h(a_k)= h(a_l)$.
\end{defi}}

If {a crossing $j$ is non-degenerate}, 
then we subdivide the octahedron { on the crossing $j$} into four tetrahedra by adding edge ${\rm E}_j{\rm F}_j$ as in Figure \ref{twistocta}(b).
Also, if {a crossing $j$ is degenerate}, then we subdivide it by adding edge ${\rm A}_j{\rm C}_j$ as in Figure \ref{twistocta}(c).
{ These subdivision guarantees non-degeneracy of all tetrahedra, which will be proved at the end of this section.
The resulting triangulation} is called {\it the octahedral triangulation} of $\mathbb{S}^3\backslash (L\cup\{\pm\infty\})$.

Consider the shadow-coloring of a link diagram $D$ induced by $\rho$, and
let $\{a_1,a_2,\ldots,a_n\}$ be the arc-colors and $\{s_1,s_2,\ldots,s_m\}$ be the region-colors.
The number of these colors is finite, so can choose an element $p\in\mathcal{P}$ satisfying 
\begin{equation}\label{defp}
h(p)\notin\{h(a_1),\ldots,h(a_n), h(s_1),\ldots,h(s_m)\}.
\end{equation}

The geometric shape of the triangulation is determined by the shadow-coloring induced by $\rho$ in the following way. 
If the crossing $j$ {in Figure \ref{crossings} is non-degenerate and} positive, then let the signed coordinates of the tetrahedra 
${\rm E}_j{\rm F}_j{\rm C}_j{\rm D}_j$, ${\rm E}_j{\rm F}_j{\rm A}_j{\rm D}_j$, 
${\rm E}_j{\rm F}_j{\rm A}_j{\rm B}_j$, ${\rm E}_j{\rm F}_j{\rm C}_j{\rm B}_j$ be
\begin{equation}\label{octa1}
(a_l, a_k, s*a_l,p), -(a_l, a_k, s, p), (a_l*a_k,a_k,s*a_k,p), -(a_l*a_k,a_k, (s*a_l)*a_k, p),
\end{equation}
respectively. Here, the minus sign of the coordinate means the orientation of the tetrahedron
does not coincide with the one induced by the vertex-ordering. 
Also, if the crossing $j$ is {non-degenerate and} negative, then 
let the signed coordinates of the tetrahedra 
${\rm E}_j{\rm F}_j{\rm C}_j{\rm D}_j$, ${\rm E}_j{\rm F}_j{\rm A}_j{\rm D}_j$, 
${\rm E}_j{\rm F}_j{\rm A}_j{\rm B}_j$, ${\rm E}_j{\rm F}_j{\rm C}_j{\rm B}_j$ be
\begin{equation}\label{octa1-1}
(a_l, a_k, s, p), -(a_l, a_k, s*a_l,p), (a_l*a_k,a_k, (s*a_l)*a_k, p), -(a_l*a_k,a_k,s*a_k,p),
\end{equation}
respectively. Figures \ref{pic08}--\ref{coordiocta} show the signed coordinates of (\ref{octa1}) and (\ref{octa1-1}).

\begin{figure}[h]
\subfigure[Positive crossing]{
  \includegraphics[scale=1]{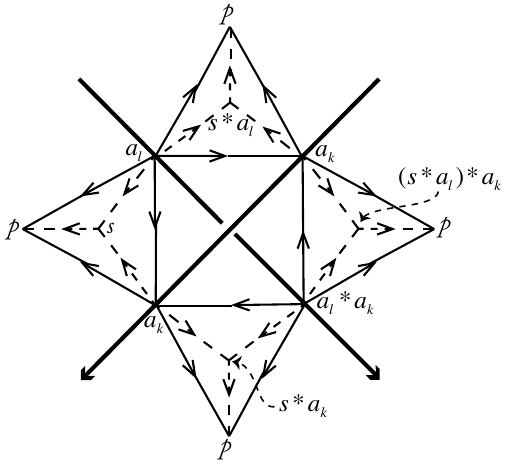}}
  \subfigure[Negative crossing]{
  \includegraphics[scale=1]{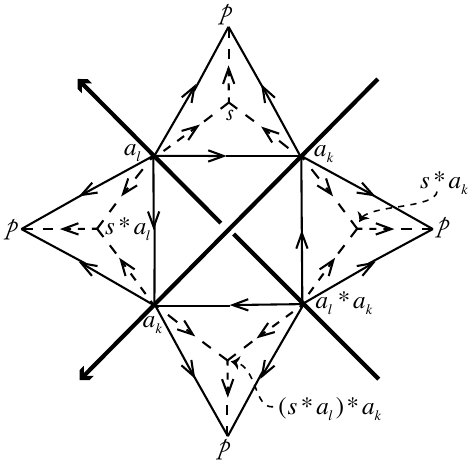} } 
  \caption{Cooridnates of tetrahedra when $h(a_k)\neq h(a_l)$}\label{pic08}
\end{figure}

\begin{figure}[h]
\subfigure[Positive crossing]{
  \includegraphics[scale=1.1]{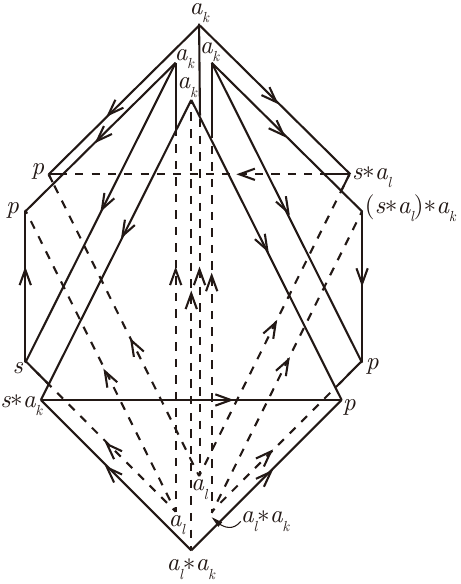}}
  \subfigure[Negative crossing]{
  \includegraphics[scale=1.1]{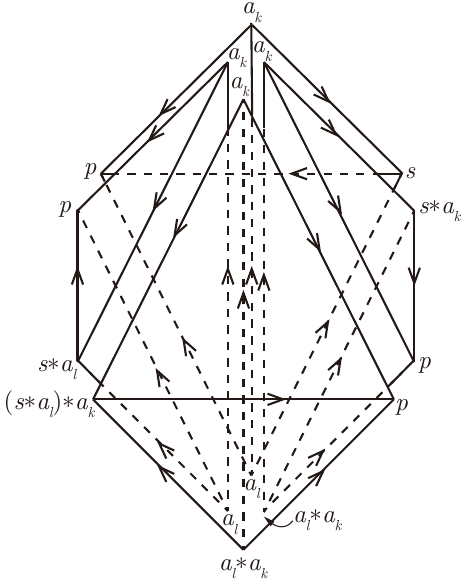} } 
  \caption{Figure \ref{pic08} in octahedral position}\label{coordiocta}
\end{figure}

On the other hand, if the crossing $j$ {in Figure \ref{crossings} is degenerate and} is positive,
then let the signed coordinates of the tetrahedra 
${\rm F}_j{\rm A}_j{\rm C}_j{\rm D}_j$, ${\rm E}_j{\rm A}_j{\rm C}_j{\rm D}_j$,
${\rm E}_j{\rm A}_j{\rm C}_j{\rm B}_j$, ${\rm F}_j{\rm A}_j{\rm C}_j{\rm B}_j$
be
\begin{equation}\label{octa2}
-(a_k,s, s*a_l,p), (a_l,s, s*a_l, p), -(a_l*a_k,s*a_k,(s*a_l)*a_k, p), (a_k,s*a_k, (s*a_l)*a_k, p),
\end{equation}
respectively. If $j$ is {degenerate and} negative, then let the signed coordinates be
\begin{equation}\label{octa2-1}
-(a_k,s*a_l,s,p), (a_l,s*a_l,s, p), -(a_l*a_k,(s*a_l)*a_k,s*a_k, p), (a_k,(s*a_l)*a_k,s*a_k, p),
\end{equation}
respectively.


Figure \ref{degocta} shows the signed coordinates of (\ref{octa2}) and (\ref{octa2-1}).
Note that the orientations of (\ref{octa1})--(\ref{octa2-1}) are different from \cite{Kabaya14} and match with \cite{Cho13a}.

\begin{figure}[h]\centering
\subfigure[Positive crossing]{
  \includegraphics[scale=0.8]{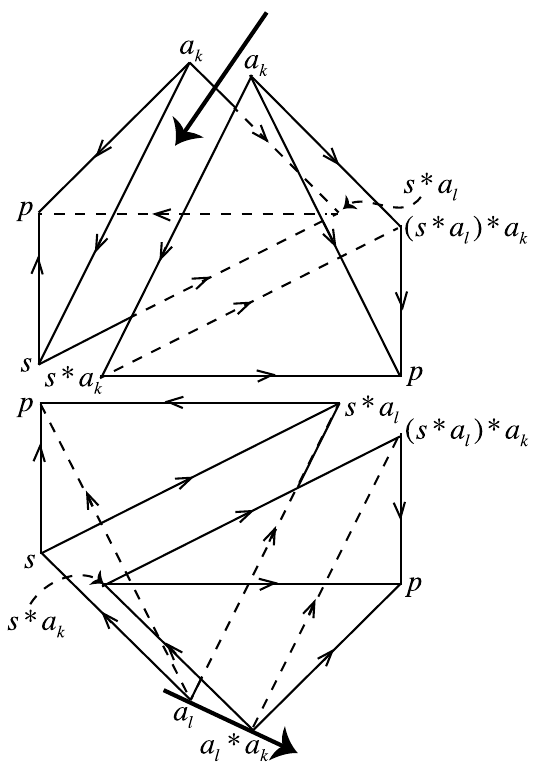}}
  \subfigure[Negative crossing]{
  \includegraphics[scale=0.8]{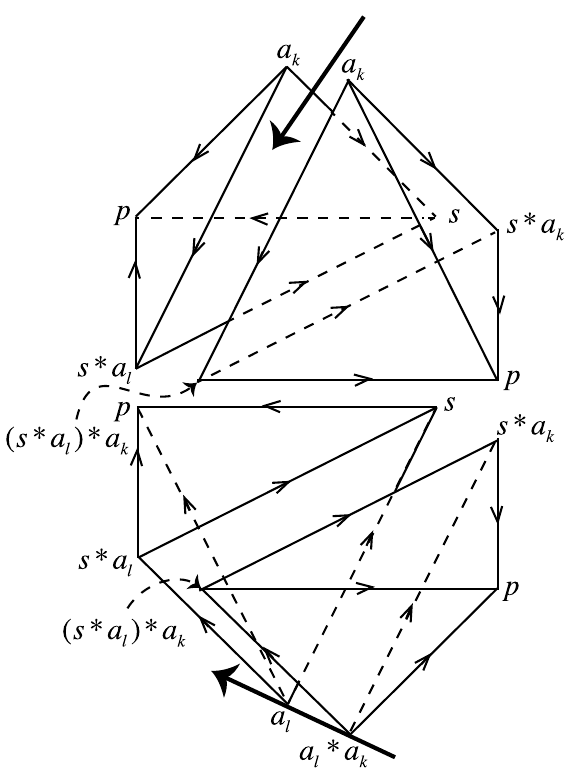} } 
  \caption{Cooridnates of tetrahedra when $h(a_k)= h(a_l)$}\label{degocta}
\end{figure}

We remark that the signed coordinates (\ref{octa1})--(\ref{octa2-1}) actually define an element in
certain simplicial quandle homology {in} \cite{Kabaya14}. Although this homology is crucial for proving
the main results of \cite{Kabaya14}, we will use their results without the homology.

\begin{defi} Let $v_0,v_1,v_2,v_3\in\mathbb{CP}^1=\mathbb{C}\cup\{\infty\}=\partial\mathbb{H}^3$.
 The hyperbolic ideal tetrahedron with signed coordinate $\sigma(v_0,v_1,v_2,v_3)$ with $\sigma\in\{\pm1\}$
 is called \textbf{degenerate} when some of the vertices $v_0,v_1,v_2,v_3$ coincide, and
\textbf{non-degenerate} when all the vertices are different.
\textbf{The cross-ratio} $[v_0,v_1,v_2,v_3]^\sigma$ of the non-degenerate signed coordinate $\sigma(v_0,v_1,v_2,v_3)$ is defined by
 $$[v_0,v_1,v_2,v_3]^\sigma=\left(\frac{v_3-v_0}{v_2-v_0}\,\frac{v_2-v_1}{v_3-v_1}\right)^\sigma
 \in\mathbb{C}\backslash\{0,1\}.$$
\end{defi}

The tetrahedra in (\ref{octa1})--(\ref{octa2-1}) have elements of the coordinates in $\mathcal{P}$. Therefore,
we need to send them to points in the boundary of the hyperbolic 3-space $\partial\mathbb{H}^3$ so as to obtain hyperbolic ideal tetrahedra.
The Hopf map $h$, defined in Definition \ref{Hopf}, plays the role. 

\begin{lem}
The images of (\ref{octa1})--(\ref{octa2-1}) under the Hopf map $h$ are non-degenerate tetrahedra. 
Specifically, if the crossing $j$ {is non-degenerate and} positive, then
  \begin{eqnarray}
  (h(a_l), h(a_k), h(s*a_l),h(p)), -(h(a_l), h(a_k), h(s), h(p)),\label{octa3}\\
  (h(a_l*a_k),h(a_k),h(s*a_k),h(p)), -( h(a_l*a_k),h(a_k), h((s*a_l)*a_k), h(p))\nonumber,
  \end{eqnarray}
and, if the crossing $j$ {is non-degenerate and} negative, then
  \begin{eqnarray}
  (h(a_l), h(a_k), h(s), h(p)), -(h(a_l), h(a_k), h(s*a_l),h(p)),\label{octa3-1}\\
  ( h(a_l*a_k),h(a_k), h((s*a_l)*a_k), h(p)), -(h(a_l*a_k),h(a_k),h(s*a_k),h(p))\nonumber,
  \end{eqnarray}
are non-degenerate hyperbolic ideal tetrahedra. 

If the crossing $j$ {is degenerate and} positive, then
  \begin{eqnarray}
  (h(a_l), h(s), h(s*a_l),h(p)), -(h(a_k), h(s), h(s*a_l), h(p)),\label{octa4}\\
  (h(a_k), h(s*a_k), h((s*a_l)*a_k), h(p)), -(h(a_l*a_k), h(s*a_k), h((s*a_l)*a_k), h(p)), \nonumber
  \end{eqnarray}
and, if the crossing $j$ {is degenerate and} negative, then
  \begin{eqnarray}
  (h(a_l), h(s*a_l),h(s), h(p)), -(h(a_k), h(s*a_l),h(s),  h(p)),\label{octa4-1}\\
  (h(a_k),  h((s*a_l)*a_k), h(s*a_k),h(p)), -(h(a_l*a_k), h((s*a_l)*a_k), h(s*a_k),h(p)), \nonumber
  \end{eqnarray}
are non-degenerate ideal hyperbolic tetrahedra.
\end{lem}

\begin{proof} Note that the {region-coloring} we are considering satisfies Lemma \ref{lem1}.
To show the non-degeneracy { of a tetrahedron}, it is enough to show any two endpoints of an edge are different.

In the cases of (\ref{octa3})--(\ref{octa3-1}), endpoints of any edge are adjacent, as a pair among $a_k, s, s*a_k$ in Figure \ref{pic04} (to check the adjacency, refer Figure \ref{pic05}), or one of them is $p$, except
the edges $(a_l, a_k)$, $(a_l*a_k,a_k)$.
Therefore, it is enough to show that $h(a_k)\neq h(a_l)$ implies $h(a_l*a_k)\neq h(a_k)$, which is trivial because
$h(a_l*a_k)=h(a_k*a_k)$ implies $h(a_l)=h(a_k)$.

In the cases of (\ref{octa4})--(\ref{octa4-1}), all endpoints of edges are adjacent or one of them is $p$, 
so we get the proof.

\end{proof}

Note that, when {the crossing $j$ is degenerate}, first two tetrahedra in (\ref{octa4}) share the same coordinate
with different signs and the others do the same. Therefore, all tetrahedra cancel out each other geometrically
and we can remove the octahedron of the crossing. 
{(This is why the crossing is called degenerate.)} Also, the same holds for (\ref{octa4-1}).
{This idea} will be used in Section \ref{sec3}.

{The assignment of the coordinates to tetrahedra above is from \cite{Kabaya14}.
Note that this assignment is based on the construction of the developing map of $\rho$
proposed in \cite{Neumann99} and \cite{Zickert09},
so the shapes of the triangulation determines the developing map of $\rho$.}

\subsection{Complex volume of $\rho$}

Consider an ideal tetrahedron with vertices $v_0$, $v_1$, $v_2$, $v_3$, where $v_k\in\mathbb{CP}^1$.
For each edge $v_k v_l$, we assign $g_{kl}$ and $\widehat{g}_{kl}\in\mathbb{CP}^1$, 
and call them {\it long-edge parameter} and {\it edge parameter}, respectively. (See Figure \ref{pic11}.)
Later, we will distinguish them by considering $g_{kl}$ is assigned to the edge of a triangulation
and $\widehat{g}_{kl}$ to the edge of a tetrahedron.

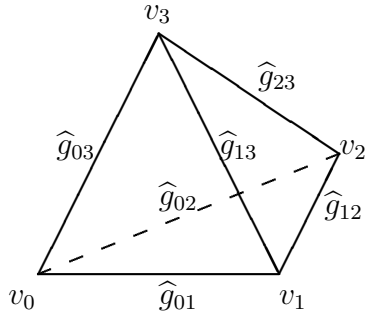
\begin{figure}[h]
\centering
  {\setlength{\unitlength}{0.4cm}
  \begin{picture}(12,10)\thicklines
   \put(1,1){\line(1,0){8}}
   \put(1,1){\line(1,2){4}}
   \put(5,9){\line(1,-2){4}}
   \put(9,1){\line(1,2){2}}
   \put(5,9){\line(3,-2){6}}
   \dashline{0.5}(1,1)(11,5)
   \put(0,0){$v_0$}
   \put(9,0){$v_1$}
   \put(11,5){$v_2$}
   \put(4.5,9.5){$v_3$}
   \put(5,0){$\widehat{g}_{01}$}
   \put(8.3,7.3){$\widehat{g}_{23}$}
   \put(1.6,5){$\widehat{g}_{03}$}
   \put(10.5,3){$\widehat{g}_{12}$}
   \put(5,3.3){$\widehat{g}_{02}$}
   \put(7,5){$\widehat{g}_{13}$}
  \end{picture}}
  \caption{Edge parameter}\label{pic11}
\end{figure}

\begin{defi} For the edge parameter $\widehat{g}_{kl}$ of an ideal tetrahedron, 
\textbf{Ptolemy {relation}} is the following equation:
$$\widehat{g}_{02}\widehat{g}_{13}=\widehat{g}_{01}\widehat{g}_{23}+\widehat{g}_{03}\widehat{g}_{12}.$$
\end{defi}

For example, if we define edge parameter $\widehat{g}_{kl}:=v_l-v_k$, then direct calculation shows
\begin{equation}\label{eq7}
  (v_2-v_0)(v_3-v_1)=(v_1-v_0)(v_3-v_2)+(v_3-v_0)(v_2-v_1),
\end{equation}
which is the Ptolemy {relation}. Furthermore, these edge parameters satisfy
\begin{equation}\label{eq8}
  [v_0,v_1,v_2,v_3]=\frac{\widehat{g}_{03}\widehat{g}_{12}}{\widehat{g}_{02}\widehat{g}_{13}}.
\end{equation}
{ }

To apply the results of \cite{Zickert09} and \cite{Hikami13}, 
{the edge parameters should satisfy the Ptolemy relation, (\ref{eq8}) and one more condition that
they} should depend on the edge of the triangulation, not of the tetrahedron.
In other words, if two edges are glued in the triangulation, {the edge parameters should be the same}. 
We call this {latter} condition {\it the coincidence  condition}.
When the edge-parameters satisfy the coincidence {condition}, we call them {\it the long-edge parameters} and denote it by $g_{kl}$. 
(We also need extra condition that the orientations of the two glued edges induced by the vertex-orientations 
of each tetrahedra should coincide. However, the vertex-orientation in (\ref{octa3})--(\ref{octa4-1}) always satisfies it.)
Unfortunately, the edge-parameter $\widehat{g}_{kl}=v_l-v_k$ defined above
does not satisfy this condition, so we will redefine the edge-parameter and the long-edge parameter using \cite{Kabaya14} as follows.

At first, consider two elements 
$a=\left(\begin{array}{cc}\alpha_1 &\alpha_2\end{array}\right), b=\left(\begin{array}{cc}\beta_1 &\beta_2\end{array}\right)$ in $\mathcal{P}$.
We define {\it determinant} $\det(a,b)$ by
\begin{equation*}
  \det(a,b):=\pm\det\left(\begin{array}{cc}\alpha_1 & \alpha_2 \\\beta_1 & \beta_2\end{array}\right)=\pm(\alpha_1 \beta_2-\alpha_2\beta_1 ).
\end{equation*}
Note that the determinant is defined up to sign due to the choice of the representative 
$a=\left(\begin{array}{cc} \alpha_1 & \alpha_2\end{array}\right)=\left(\begin{array}{cc} -\alpha_1 & -\alpha_2\end{array}\right)\in\mathcal{P}$.
To remove this ambiguity, we fix representatives\footnote{
The difference with \cite{Kabaya14} is that they chose
a sign of the determinant once and for all. Their choice is good enough to define long-edge parameter $g_{jk}$, 
but not for edge parameter $\widehat{g}_{jk}$.} 
of arc-colors in $\mathbb{C}^2\backslash\{0\}$ once and for all. 
Then we fix a representative of one region-color, which uniquely determines the representatives of all the other region-colors 
by the arc-coloring.
(This is due to the fact that $s*(\pm a)=s*a$ for any $s,a\in\mathbb{C}^2\backslash\{0\}$.)

After fixing all the representatives of the shadow-coloring, we obtain a well-defined determinant
\begin{equation}\label{det}
  \det(a,b)=\det\left(\begin{array}{cc}\alpha_1 & \alpha_2 \\\beta_1 & \beta_2\end{array}\right)=\alpha_1 \beta_2- \alpha_2\beta_1.
\end{equation}

\begin{lem} For $a,b,c\in\mathbb{C}^2\backslash\{0\}$, the determinant satisfies
$$\det(a*c,b*c)=\det(a,b).$$
\end{lem}
\begin{proof}
  Let $a=\left(\begin{array}{cc} \alpha_1 & \alpha_2\end{array}\right)$, $b=\left(\begin{array}{cc} \beta_1 & \beta_2\end{array}\right)$,
  $c=\left(\begin{array}{cc} \gamma_1 & \gamma_2\end{array}\right)$,
  and $C=\left(\begin{array}{cc} 1+\gamma_1 \gamma_2& \gamma_2^2 \\ -\gamma_1^2& 1-\gamma_1 \gamma_2\end{array}\right)$. Then
  $$\det(a*c,b*c)=\det(aC,bC)
  =\det(a,b)\cdot \det C=\det(a,b).$$
\end{proof}

Consider the shadow-coloring and the coordinates of tetrahedra in Figure \ref{pic08} (or Figure \ref{coordiocta}) and Figure \ref{degocta}. 
We define the edge parameter $\widehat{g}_{kl}$
using those coordinates. Specifically, when the signed coordinate of the tetrahedron is $\sigma(a_0,a_1,a_2,a_3)$ 
with $\sigma\in\{\pm 1\}$ and $a_k\in\mathbb{C}^2\backslash\{0\}$, we define the edge parameter by
\begin{equation}\label{gjk}
  \widehat{g}_{kl}=\det(a_k,a_l).
\end{equation}
For example, the edge parameters of the tetrahedron $\mp(a_l,a_k,s,p)$ in the left-hand or the right-hand side of Figure \ref{pic08}
(or Figure \ref{coordiocta}) are defined by
\begin{eqnarray*}
  \widehat{g}_{01}=\det(a_l,a_k), &\widehat{g}_{02}=\det(a_l,s), &\widehat{g}_{03}=\det(a_l,p), \\
  \widehat{g}_{12}=\det(a_k,s), &\widehat{g}_{13}=\det(a_k,p), &\widehat{g}_{23}=\det(s,p).
\end{eqnarray*}

\begin{lem} The edge parameter $\widehat{g}_{kl}$ of the tetrahedron $\sigma(a_0,a_1,a_2,a_3)$ 
defined in (\ref{gjk}) satisfies the Ptolemy identity and
  \begin{equation}\label{eq11}
    [h(a_0),h(a_1),h(a_2),h(a_3)]=\frac{\widehat{g}_{03}\widehat{g}_{12}}{\widehat{g}_{02}\widehat{g}_{13}}.
  \end{equation}
\end{lem}

\begin{proof} From (\ref{det}), we obtain
  \begin{equation}\label{eq12}
  h(x)-h(y)=\frac{x_1}{x_2}-\frac{y_1}{y_2}=\frac{\det(x,y)}{x_2 y_2},
  \end{equation}
where $x=\left(\begin{array}{cc} x_1 & x_2\end{array}\right)$ and $y=\left(\begin{array}{cc} y_1 & y_2\end{array}\right)$.

Let $a_k=\left(\begin{array}{cc} \alpha_k & \beta_k\end{array}\right)$ for $k=0,\ldots,3$, and
let $v_k=h(a_k)=\frac{\alpha_k}{\beta_k}$. Then (\ref{eq7}) and (\ref{eq12}) imply
$$\frac{\det(a_0,a_2)}{\beta_0\beta_2}\frac{\det(a_1,a_3)}{\beta_1\beta_3}
=\frac{\det(a_0,a_1)}{\beta_0\beta_1}\frac{\det(a_2,a_3)}{\beta_2\beta_3}
+\frac{\det(a_0,a_3)}{\beta_0\beta_3}\frac{\det(a_1,a_2)}{\beta_1\beta_2},$$
which is equivalent to the Ptolemy identity 
$\widehat{g}_{02}\widehat{g}_{13}=\widehat{g}_{01}\widehat{g}_{23}+\widehat{g}_{03}\widehat{g}_{12}$.

Also, using (\ref{eq12}), we obtain
  \begin{eqnarray*}
    [h(a_0),h(a_1),h(a_2),h(a_3)]=\frac{\frac{\det(a_0,a_3)}{\beta_0\beta_3}}{\frac{\det(a_1,a_3)}{\beta_1\beta_3}}
    \frac{\frac{\det(a_1,a_2)}{\beta_1\beta_2}}{\frac{\det(a_0,a_2)}{\beta_0\beta_2}}
    =\frac{\widehat{g}_{03}\widehat{g}_{12}}{\widehat{g}_{02}\widehat{g}_{13}}.
  \end{eqnarray*}

\end{proof}

Note that, by the same calculation of the proof above, we obtain
$$[h(a_0),h(a_3),h(a_1),h(a_2)]=\frac{\widehat{g}_{02}\widehat{g}_{13}}{\widehat{g}_{01}\widehat{g}_{23}},~
[h(a_0),h(a_2),h(a_3),h(a_1)]=-\frac{\widehat{g}_{01}\widehat{g}_{23}}{\widehat{g}_{03}\widehat{g}_{12}}.$$
If we put $z^\sigma=[h(a_0),h(a_1),h(a_2),h(a_3)]$, using Ptolemy identity, the above equations are expressed by
\begin{equation}\label{edge_para}
  z^\sigma=\frac{\widehat{g}_{03}\widehat{g}_{12}}{\widehat{g}_{02}\widehat{g}_{13}},
  ~\frac{1}{1-z^\sigma}=\frac{\widehat{g}_{02}\widehat{g}_{13}}{\widehat{g}_{01}\widehat{g}_{23}},
  ~1-\frac{1}{z^\sigma}=-\frac{\widehat{g}_{01}\widehat{g}_{23}}{\widehat{g}_{03}\widehat{g}_{12}}.
\end{equation}

The edge parameter $\widehat{g}_{jk}$ defined above satisfies all needed properties of the long-edge parameter $g_{jk}$ except
{\it the coincidence }, which $\widehat{g}_{jk}$ satisfies up to sign.
To see this phenomenon, consider the two edges of Figure \ref{pic08}(a) as in Figure \ref{inconsistency}, which are glued in the triangulation.
Assume the chosen representative of $a_m$ in Figure \ref{inconsistency} satisfies $a_m=-a_l*a_k\in\mathbb{C}^2\backslash\{0\}$.
(This actually happens often and quite important. For example, the minus signs of (\ref{minus1}) and (\ref{minus2}) in Section \ref{sec4}
show this situation. It will be discussed seriously at later article.) Then the edge parameters satisfy
$$\widehat{g}_{01}=\det(a_l,a_k)=\det(a_l*a_k,a_k)=-\det(a_m,a_k)=-\widehat{g}_{01}'.$$

\begin{figure}[h]
\centering  \setlength{\unitlength}{0.9cm}\thicklines
\begin{picture}(6,4)  
    \put(4,4){\vector(-1,-1){4}}
    \put(0,4){\line(1,-1){1.8}}
    \put(2.2,1.8){\vector(1,-1){1.8}}
    \put(1,3){\vector(0,-1){2}}
    \put(1,3.2){$a_l$}
    \put(1,0.7){$a_k$}
    \put(0.4,2){$\widehat{g}_{01}$}
    \put(3,1){\vector(0,1){2}}
    \put(3.2,2.9){$a_k$}
    \put(3.2,0.9){$a_m=-a_l*a_k$}
    \put(3.1,2){$\widehat{g}_{01}^\prime$}
  \end{picture}
  \caption{Example of the inconsistency of edge parameter}\label{inconsistency}
\end{figure}
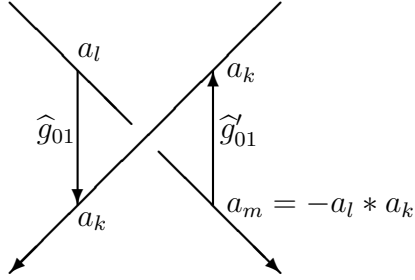

To obtain the long-edge parameter $g_{jk}$, we assign certain signs to the edge parameters
$$g_{jk}=\pm \widehat{g}_{jk},$$
so that the consistency property holds. Due to Lemma 6 of \cite{Kabaya14},
any choice of values of ${g}_{jk}$ determines the same complex volume.
Actually, in Section \ref{sec3}, we do not need the exact values of $g_{jk}$,
but we use the existence of them.

The relations (\ref{edge_para}) of the edge parameters become
\begin{equation}\label{eq13}
  z^\sigma=\pm\frac{g_{03}{g}_{12}}{{g}_{02}{g}_{13}},
  ~\frac{1}{1-z^\sigma}=\pm\frac{{g}_{02}{g}_{13}}{{g}_{01}{g}_{23}},
  ~1-\frac{1}{z^\sigma}=\pm\frac{{g}_{01}{g}_{23}}{{g}_{03}{g}_{12}}.
\end{equation}
Using (\ref{eq13}), we define integers $p$ and $q$ by
  \begin{eqnarray}\label{pq}
      \left\{\begin{array}{ll}
      p\pi i =-\log z^{\sigma}+\log g_{03}+\log g_{12}-\log g_{02}-\log g_{13},\\
      q\pi i =\log(1-z^{\sigma})+\log g_{02}+\log g_{13}-\log g_{01}-\log g_{23}.
      \end{array}\right.
  \end{eqnarray}
Now we consider the tetrahedron with the signed coordinate $\sigma(a_0,a_1,a_2,a_3)$ and the signed triples $\sigma[z^\sigma;p,q]\in\widehat{\mathcal{P}}(\mathbb{C})$.
({\it The extended pre-Bloch group} is denoted by $\widehat{\mathcal{P}}(\mathbb{C})$ here. For the definition, see Definition 1.6 of \cite{Zickert09}.)
To consider all signed triples corresponding to all tetrahedra in the triangulation, we denote the triple by $\sigma_t[z_t^{\sigma_t};p_t,q_t]$, where $t$ is the index of tetrahedra.
We define a function $\widehat{L}:\widehat{\mathcal{P}}(\mathbb{C})\rightarrow \mathbb{C}/\pi^2\mathbb{Z}$ by
\begin{equation}\label{dilo}
  [z;p,q] \mapsto \li(z)+\frac{1}{2}\log z\log(1-z)+\frac{\pi i}{2}(q\log z+p\log(1-z))-\frac{\pi^2}{6},
\end{equation}
where $\li(z)=-\int_0^z\frac{\log(1-t)}{t}dt$ is the dilogarithm function. (Well-definedness of $\widehat{L}$ was proved in \cite{Neumann04}.)
Recall that, for a boundary-parabolic representation $\rho$, the hyperbolic volume $\vol(\rho)$ and the Chern-Simons invariant $\cs(\rho)$ 
was already defined in \cite{Zickert09}. We call $\vol(\rho)+i\,\cs(\rho)$ {\it the complex volume of $\rho$}.
The following theorem is one of the main result of \cite{Kabaya14}.

\begin{thm}[\cite{Zickert09}, \cite{Kabaya14}]\label{thmvol}
For a given boundary-parabolic representation $\rho$ and the shadow-coloring induced by $\rho$,
the complex volume of $\rho$ is calculated by
$$\sum_t \sigma_t\,\widehat{L}[z_t^{\sigma_t};p_t,q_t]\equiv i(\vol(\rho)+i\,\cs(\rho))\modulo,$$
where $t$ is over all tetrahedra of the triangulation defined in Section \ref{triang}.

\end{thm}

\begin{proof}
See Theorem 5 of \cite{Kabaya14}. 

\end{proof}

Note that the removal of the tetrahedra in (\ref{octa4}) and (\ref{octa4-1}) does not have any effect on the complex volume. 
For example, if we put $[z;p,q]$ and $-[z';p',q']$ the corresponding triples of the tetrahedron 
$(h(a_l),h(s),h(s*a_l),h(p))$ and $-(h(a_k),h(s),h(s*a_l),h(p))$ in (\ref{octa4}), respectively,
and put $\{{g}_{kl}\}, \{{g}_{kl}'\}$ the sets of long-edge parameters 
of the two tetrahedra, respectively. Then, from $h(a_l)=h(a_k)$, we obtain $z=z'$. 
Furthermore, we can choose long-edge parameters so that 
${g}_{kl}={g}_{kl}'$ holds for all pairs of edges sharing the same coordinate, 
which induces $p=p'$, $q= q'$ and $\widehat{L}[z;p,q]-\widehat{L}[z';p',q']=0$.

\section{Optimistic limit}\label{sec3}

In this Section, we will use the result of Section \ref{sec2} to redefine the optimistic limit of \cite{Cho13a} and 
construct a solution of $\mathcal{H}$.
At first, we consider a given boundary-parabolic representation $\rho$ and fix its shadow-coloring of a link diagram $D$. For the diagram, define
\text{sides} of the diagram by the lines connecting two adjacent crossings. 
(The word {\it edge} is more common than {\it side} here. However, we want to keep the word {\it edge} for the edges of a triangulation.)
For example, the diagram in Figure \ref{pic13} has eight sides. 
We assign $z_1,\ldots,z_n$ to sides of $D$ as in Figure \ref{pic13} and call them {\it side variables}.

\begin{figure}[h]
\centering
  \includegraphics[scale=0.45]{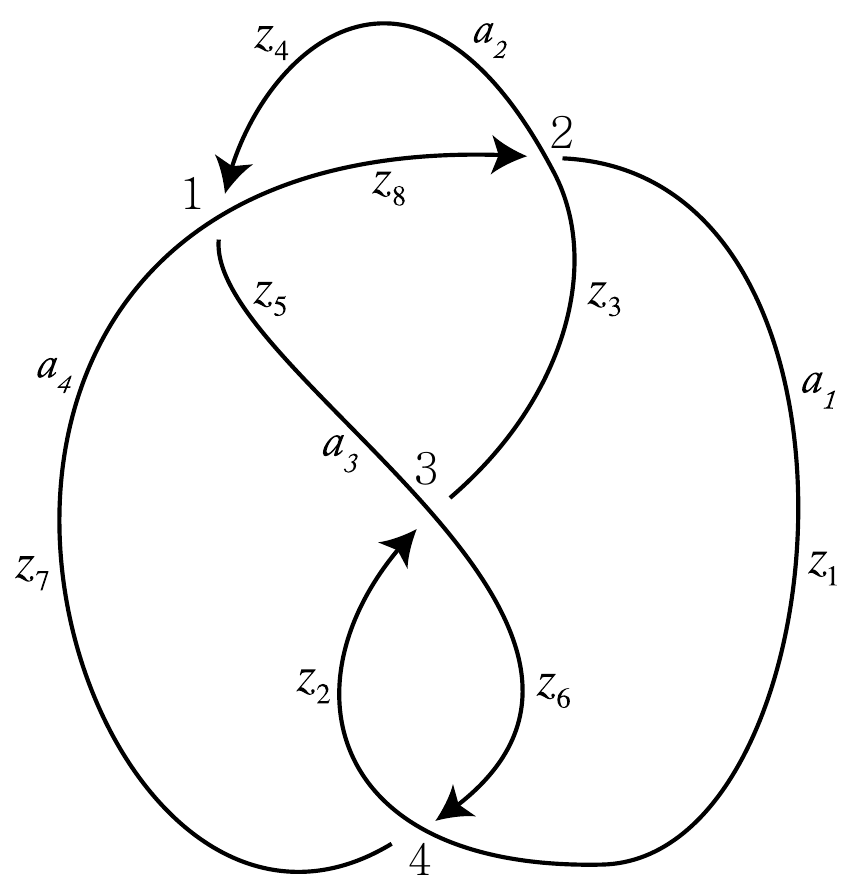}
  \caption{Sides of a link diagram}\label{pic13}
\end{figure}

For the crossing $j$ in Figure \ref{pic14}, let $z_e,z_f,z_g,z_h$ be side variables and let $a_l, a_k$ be the arc-colors.
If $h(a_k)\neq h(a_l)$, then we define the potential function $V_j$ of the crossing $j$ by
  \begin{equation}\label{potential1}
      V_j(z_e,z_f,z_g,z_h)=\li(\frac{z_f}{z_e})-\li(\frac{z_f}{z_g})+\li(\frac{z_h}{z_g})-\li(\frac{z_h}{z_e}).
   \end{equation}
   
\begin{figure}[h]
\centering  \setlength{\unitlength}{0.8cm}
\begin{picture}(5.5,4.5)\thicklines
    \put(5,4){\vector(-1,-1){4}}
    \put(2.8,2.2){\line(-1,1){1.8}}
    \put(3.2,1.8){\line(1,-1){1.8}}
    \put(5,4.2){$a_k$}
    \put(0.5,4.2){$a_l$}
    \put(2.8,1.3){$j$}
    \put(1.2,1){$z_e$}
    \put(4.3,1){$z_f$}
    \put(4.3,2.9){$z_g$}
    \put(1.2,2.9){$z_h$}
        \put(3,2){\arc[5](1,1){360}}
    \put(1.8,0.3){${\rm A}_j$}
    \put(3.8,0.3){${\rm B}_j$}
    \put(3.6,3.5){${\rm C}_j$}
    \put(1.6,3.5){${\rm D}_j$}
  \end{picture}
  \caption{A crossing $j$ with arc-colors and side variables}\label{pic14}
\end{figure}
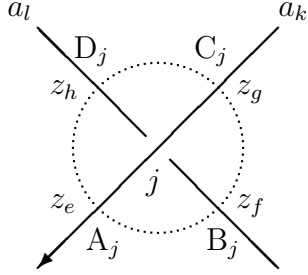

On the other hand, if $h(a_l)= h(a_k)$ in Figure \ref{pic14}, then we introduce new variables $w_e^j,w_f^j,w_g^j$ of the crossing $j$ and define 
\begin{eqnarray}
\lefteqn{V_j(z_e,z_f,z_g,z_h,w_e^j,w_f^j,w_g^j)}\label{potential2}\\
&&=-\log w_e^j \log z_e+\log w_f^j \log z_f-\log w_g^j \log z_g+\log\frac{w_e^j  w_g^j}{w_f^j}\log z_h.\nonumber
\end{eqnarray} 
For notational convenience, we put $w_h^j:={w_e^j  w_g^j}/w_f^j$.
(In (\ref{potential2}), we can choose any three variables among $w_e^j,w_f^j,w_g^j,w_h^j$ free variables.)
We call the crossing $j$ in Figure \ref{pic14} {\it degenerate} when $h(a_l)= h(a_k)$ holds.
In particular, when the degenerate crossing forms a kink, as in Figure \ref{kink}, we put
\begin{eqnarray*}
\lefteqn{V_j(z_e,z_f,z_g,w_e^j,w_f^j)}\\
&&=-\log w_e^j \log z_e+\log w_f^j \log z_f-\log w_f^j \log z_f+\log\frac{w_e^j  w_f^j}{w_f^j}\log z_g\\
&&=-\log w_e^j\log z_e+\log w_e^j \log z_g.
\end{eqnarray*} 

\begin{figure}[h]
\centering
  \includegraphics[scale=1.5]{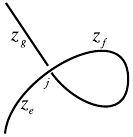}
  \caption{Kink}\label{kink}
\end{figure}

Consider the crossing $j$ in Figure \ref{pic14} and place the octahedron 
${\rm A}_j{\rm B}_j{\rm C}_j{\rm D}_j{\rm E}_j{\rm F}_j$ as in Figure \ref{twistocta}.
When the crossing $j$ is non-degenerate, in other words $h(a_k)\neq h(a_l)$, 
we consider Figure \ref{twistocta}(b) and assign shape parameters $\frac{z_f}{z_e}$, $\frac{z_g}{z_f}$, $\frac{z_h}{z_g}$ and $\frac{z_e}{z_h}$
to the horizontal edges ${\rm A}_j{\rm B}_j$, ${\rm B}_j{\rm C}_j$, ${\rm C}_j{\rm D}_j$, ${\rm D}_j{\rm A}_j$, respectively.
On the other hand, if the crossing $j$ is degenerate, in other words $h(a_k)= h(a_l)$, then we consider Figure \ref{twistocta}(c) and assign shape parameters $w_e^j,w_f^j,w_g^j$ and $w_h^j$ to
the edges ${\rm A}_j{\rm F}_j$, ${\rm B}_j{\rm E}_j$, ${\rm C}_j{\rm F}_j$ and ${\rm D}_j{\rm E}_j$, respectively.\footnote{
Note that, when $h(a_k)= h(a_l)$, by adding one more edge ${\rm B}_j{\rm D}_j$ to Figure \ref{twistocta}(c),
we obtain another subdivision of the octahedron with five tetrahedra. (This subdivision was already used in \cite{Cho13c}.)
Focusing on the middle tetrahedron  
that contains all horizontal edges, we obtain $w_e^j w_g^j =w_f^j w_h^j$. Furthermore, the shape-parameters assigned to
${\rm D}_j{\rm F}_j$ and ${\rm B}_j{\rm F}_j$ are
$\frac{1-1/w_e^j}{1-w_g^j}$ and $\frac{1-1/w_g^j}{1-w_e^j}$, respectively.}

The {\it potential function} $V(z_1,\ldots,z_n,w_k^j,\ldots)$ of the link diagram $D$ is defined by
\begin{equation*}
  V(z_1,\ldots,z_n,w_k^j,\ldots)=\sum_j V_j,
\end{equation*}
where $j$ is over all crossings. For example, if $h(a_1)\neq h(a_2)$ in Figure \ref{pic13}, then
$a_4=a_1*a_2$ implies\footnote{
If $h(a_4)=h(a_2)$, then $h(a_2*a_2)=h(a_2)=h(a_4)=h(a_1*a_2)$ induces $h(a_2)=h(a_1)$,
which is contradiction.}
 $h(a_4)\neq h(a_2)$, $a_2=a_1*a_3$ does\footnote{
If $h(a_2)=h(a_3)$, then $h(a_3*a_3)=h(a_3)=h(a_2)=h(a_1*a_3)$ induces $h(a_2)=h(a_3)=h(a_1)$,
which is contradiction. Likewise, if $h(a_1)=h(a_3)$, then $h(a_2)=h(a_1*a_3)=h(a_1)$ is contradiction.}
 $ h(a_2)\neq h(a_3)\neq h(a_1)$,
$a_2=a_3*a_4$ does $h(a_4)\neq h(a_3)$, $a_4=a_3*a_1$ does $h(a_4)\neq h(a_1)$,
and the potential function becomes
\begin{eqnarray}
  V(z_1,\ldots,z_8)&=&\left\{\li(\frac{z_5}{z_7})-\li(\frac{z_5}{z_8})+\li(\frac{z_4}{z_8})-\li(\frac{z_4}{z_7})\right\}\label{V41}\\
           &+&\left\{\li(\frac{z_1}{z_3})-\li(\frac{z_1}{z_4})+\li(\frac{z_8}{z_4})-\li(\frac{z_8}{z_3})\right\}\nonumber\\
           &+&\left\{\li(\frac{z_3}{z_6})-\li(\frac{z_3}{z_5})+\li(\frac{z_2}{z_5})-\li(\frac{z_2}{z_6})\right\}\nonumber\\
           &+&\left\{\li(\frac{z_6}{z_1})-\li(\frac{z_6}{z_2})+\li(\frac{z_7}{z_2})-\li(\frac{z_7}{z_1})\right\}.\nonumber
\end{eqnarray}
Note that, if $h(a_l)\neq h(a_k)$ for any crossing $j$ in Figure \ref{pic14}, then the definition of the potential function above
coincides with the definition in Section 2 of \cite{Cho13a}. Therefore, the above definition is a slight modification of the previous one.

On the other hand, if $h(a_1)=h(a_2)$ in Figure \ref{pic13}, then $a_1*a_2=a_1$.
This equation and the relations at crossings induce\footnote{
The relation $a_4=a_1*a_2$ induces $a_4=a_1$, $a_4=a_3*a_1$ does $a_4=a_3$, and $a_2=a_3*a_4$ does $a_2=a_4$.} 
$a_1=a_2=a_3=a_4$, 
and the potential function becomes
\begin{eqnarray*}
  \lefteqn{V(z_1,\ldots,z_8,w_8^1,w_4^1,w_7^1, w_4^2,w_8^2,w_3^2, w_6^3,w_3^3,w_5^3, w_2^4,w_7^4,w_1^4)}\\
    &&=-\log w_8^1\log z_8 +\log w_4^1\log z_4-\log w_7^1\log z_7 +\log w_5^1\log z_5\\
    && ~~- \log w_4^2\log z_4 +\log w_8^2\log z_8 -\log w_3^2 \log z_3 +\log w_1^2\log z_1\\
    &&~~-\log w_6^3 \log z_6 +\log w_3^3 \log z_3-\log w_5^3 \log z_5+\log w_2^3\log z_2\\
    &&~~-\log w_2^4 \log z_2+\log w_7^4 \log z_7-\log w_1^4 \log z_1+\log w_6^4\log z_6 ,
\end{eqnarray*}
where $w_5^1=w_8^1  w_7^1/w_4^1$, $w_1^2=w_4^2  w_3^2/w_8^2$,
$w_2^3=w_6^3  w_5^3/w_3^3$ and $w_6^4=w_2^4  w_1^4/w_7^4$.

For the potential function $V(z_1,\ldots,z_n,w_k^j,\ldots)$, let $\mathcal{H}$ be the set of equations
\begin{equation}\label{H}
\mathcal{H}:=\left\{\left.\exp(z_k \frac{\partial V}{\partial z_k})=1,
  \exp(w_k^j\frac{\partial V}{\partial w_k^j})=1\right| k=1,\ldots,n,\,j:\text{degenerate}\right\},
\end{equation}
and $\mathcal{S}=\{(z_1,\ldots,z_n,w_k^j,\ldots)\}$ be the {solution} set of $\mathcal{H}$.
Here, solutions are assumed to satisfy the properties that $z_k\neq 0$ for all $k=1,\ldots, n$ and
$\frac{z_f}{z_e}\neq 1$, $\frac{z_g}{z_f}\neq 1$, $\frac{z_h}{z_g}\neq 1$, 
$\frac{z_e}{z_h}\neq 1$, $\frac{z_g}{z_e}\neq 1$, $\frac{z_h}{z_f}\neq 1$ in Figure \ref{pic14} for any non-degenerate crossing,
and $w_k^j\neq 0$ for any degenerate crossing $j$ and the index $k$.
(All these assumptions are essential to avoid singularity of the equations in $\mathcal{H}$ and
$\log 0$ in the formula $V_0$ defined in (\ref{V0}). Even though we allow $w_k^j=1$ here,
the value we are interested in always satisfies $w_k^j\neq 1$.)

\begin{pro}\label{pro1} For the arc-coloring of a link diagram $D$ induced by $\rho$ and the potential function $V(z_1,\ldots,z_n,w_k^j,\ldots)$, 
the set $\mathcal{H}$ induces the whole set of hyperbolicity equations of the octahedral triangulation defined in Section \ref{triang}.
\end{pro}

The {\it hyperbolicity equations} consist of the Thurston's gluing equations of edges and the completeness condition.

\begin{proof}[Proof of Proposition \ref{pro1}]
When no crossing is degenerate, this proposition was already proved in Section 3 of \cite{Cho13a}. 
To see the main idea, check Figures 10--13 and equations (3.1)--(3.3) of \cite{Cho13a}.
Equation (3.1) is a completeness condition along a meridian of certain annulus, and (3.2)--(3.3) are gluing equations of certain edges.
These three types of equations induce all the other gluing equations.

Therefore, we consider the case when the crossing $j$ in Figure \ref{pic14} is degenerate. Then, the following three equations
\begin{equation}\label{partialwj}
\exp(w_e^j\frac{\partial V}{\partial w_e^j})=\frac{z_h}{z_e}=1,~
\exp(w_f^j\frac{\partial V}{\partial w_f^j})=\frac{z_f}{z_h}=1,~
\exp(w_g^j\frac{\partial V}{\partial w_g^j})=\frac{z_h}{z_g}=1
\end{equation}
induce $z_e= z_f= z_g= z_h$. This guarantees the gluing equations of horizontal edges trivially by the assigning rule of
shape parameters.
(Note that the shape parameters assigned to the horizontal edges of the octahedron at a degenerate crossing are always 1.)

There are four possible cases of gluing pattern as in Figure \ref{cases}, 
and we assume the crossing $j$ is degenerate and $j+1$ is non-degenerate.
(The case when both of $j$ and $j+1$ are degenerate can be proved similarly.)

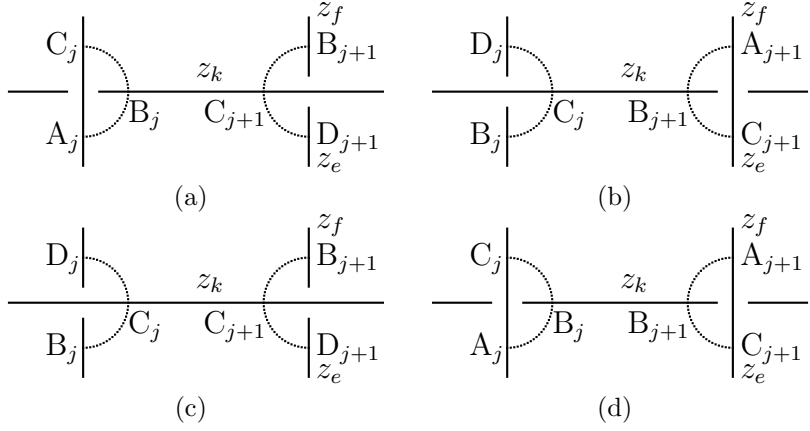
\begin{figure}[h]
\centering
  \subfigure[]
  {\begin{picture}(5,2)\thicklines
   \put(1,1){\arc[5](0,-0.6){180}}
   \put(4,1){\arc[5](0,0.6){180}}
   \put(1.2,1){\line(1,0){3.8}}
   \put(1,2){\line(0,-1){2}}
   \put(4,0){\line(0,1){0.8}}
   \put(4,2){\line(0,-1){0.8}}
   \put(0.8,1){\line(-1,0){0.8}}
   \put(0.5,0.3){${\rm A}_j$}
   \put(1.6,0.6){${\rm B}_j$}
   \put(0.5,1.5){${\rm C}_j$}
   \put(4.1,0.3){${\rm D}_{j+1}$}
   \put(2.6,0.6){${\rm C}_{j+1}$}
   \put(4.1,1.5){${\rm B}_{j+1}$}
   \put(2.5,1.2){$z_k$}
   \put(4.1,0){$z_e$}
   \put(4.1,2){$z_f$}
  \end{picture}}\hspace{0.5cm}
  \subfigure[]
  {\begin{picture}(5,2)\thicklines
   \put(1,1){\arc[5](0,-0.6){180}}
   \put(4,1){\arc[5](0,0.6){180}}
   \put(0,1){\line(1,0){3.8}}
   \put(4.2,1){\line(1,0){0.8}}
   \put(1,2){\line(0,-1){0.8}}
   \put(1,0){\line(0,1){0.8}}
   \put(4,0){\line(0,1){2}}
   \put(0.5,0.3){${\rm B}_j$}
   \put(1.6,0.6){${\rm C}_j$}
   \put(0.5,1.5){${\rm D}_j$}
   \put(4.1,0.3){${\rm C}_{j+1}$}
   \put(2.6,0.6){${\rm B}_{j+1}$}
   \put(4.1,1.5){${\rm A}_{j+1}$}
   \put(2.5,1.2){$z_k$}
   \put(4.1,0){$z_e$}
   \put(4.1,2){$z_f$}
  \end{picture}}\\
  \subfigure[]
  {\begin{picture}(5,2)\thicklines
   \put(1,1){\arc[5](0,-0.6){180}}
   \put(4,1){\arc[5](0,0.6){180}}
   \put(5,1){\line(-1,0){5}}
   \put(1,2){\line(0,-1){0.8}}
   \put(1,0){\line(0,1){0.8}}
   \put(4,2){\line(0,-1){0.8}}
   \put(4,0){\line(0,1){0.8}}
   \put(0.5,0.3){${\rm B}_j$}
   \put(1.6,0.6){${\rm C}_j$}
   \put(0.5,1.5){${\rm D}_j$}
   \put(4.1,0.3){${\rm D}_{j+1}$}
   \put(2.6,0.6){${\rm C}_{j+1}$}
   \put(4.1,1.5){${\rm B}_{j+1}$}
   \put(2.5,1.2){$z_k$}
   \put(4.1,0){$z_e$}
   \put(4.1,2){$z_f$}
  \end{picture}}\hspace{0.5cm}
  \subfigure[]
  {\begin{picture}(5,2)\thicklines
   \put(1,1){\arc[5](0,-0.6){180}}
   \put(4,1){\arc[5](0,0.6){180}}
   \put(4.2,1){\line(1,0){0.8}}
   \put(1,2){\line(0,-1){2}}
   \put(0.8,1){\line(-1,0){0.8}}
   \put(4,2){\line(0,-1){2}}
   \put(1.2,1){\line(1,0){2.6}}
   \put(0.5,0.3){${\rm A}_j$}
   \put(1.6,0.6){${\rm B}_j$}
   \put(0.5,1.5){${\rm C}_j$}
   \put(4.1,0.3){${\rm C}_{j+1}$}
   \put(2.6,0.6){${\rm B}_{j+1}$}
   \put(4.1,1.5){${\rm A}_{j+1}$}
   \put(2.5,1.2){$z_k$}
   \put(4.1,0){$z_e$}
   \put(4.1,2){$z_f$}
  \end{picture}}
  \caption{Four cases of gluing pattern}\label{cases}
\end{figure} 

The part of the potential function $V$ containing $z_k$ in Figure \ref{cases}(a) is
\begin{equation*}
  V^{(a)}=\log w_k^j\log z_k+\li\left(\frac{z_e}{z_k}\right)-\li\left(\frac{z_f}{z_k}\right),
\end{equation*}
and
\begin{equation*}
  \exp\left(z_k\frac{\partial V}{\partial z_k}\right)=
  \exp\left(z_k\frac{\partial V^{(a)}}{\partial z_k}\right)={w_k^j}\left(1-\frac{z_e}{z_k}\right)\left(1-\frac{z_f}{z_k}\right)^{-1}=1
\end{equation*}
is equivalent with the following completeness condition 
$$\frac{1}{w_k^j}\left(1-\frac{z_e}{z_k}\right)^{-1}\left(1-\frac{z_f}{z_k}\right)=1$$
along a meridian $m$ in Figure \ref{cases2}(a).
(Compare it with Figure 11 of \cite{Cho13a}.) 
Here, $a_j$, $b_j$, $c_j$, $b_{j+1}$, $c_{j+1}$, $d_{j+1}$ in Figure \ref{cases2}(a) are the points of the cusp diagram, 
which lie on the edges ${\rm A}_j{\rm E}_j$, ${\rm B}_j{\rm E}_j$, ${\rm C}_j{\rm E}_j$,
${\rm B}_{j+1}{\rm F}_{j+1}$, ${\rm C}_{j+1}{\rm F}_{j+1}$, ${\rm D}_{j+1}{\rm F}_{j+1}$
of Figure \ref{twistocta}(a), respectively.

\begin{figure}[h]
\centering
  \subfigure[]{\includegraphics[scale=0.72]{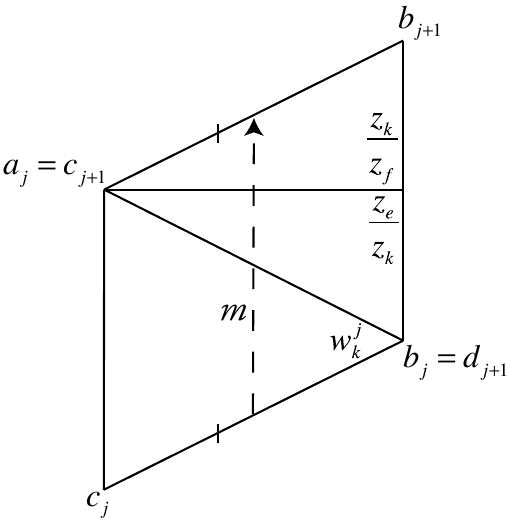}}
  \subfigure[]{\includegraphics[scale=0.72]{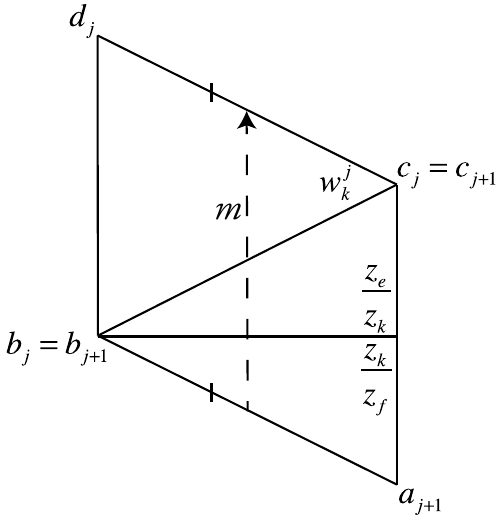}}\\
  \subfigure[]{\includegraphics[scale=0.75]{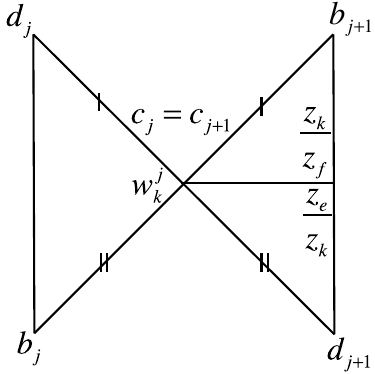}}\hspace{1cm}
  \subfigure[]{\includegraphics[scale=0.76]{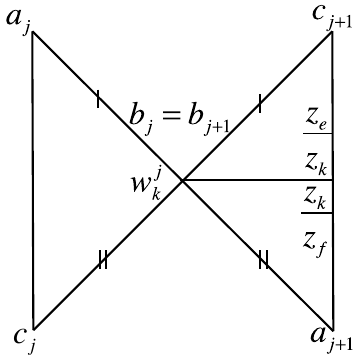}}
  \caption{Four cusp diagrams from Figure \ref{cases}}\label{cases2}
\end{figure}

The part of the potential function $V$ containing $z_k$ in Figure \ref{cases}(b) is
\begin{equation*}
  V^{(b)}=-\log w_k^j\log z_k-\li\left(\frac{z_k}{z_e}\right)+\li\left(\frac{z_k}{z_f}\right),
\end{equation*}
and
\begin{equation*}
  \exp\left(z_k\frac{\partial V}{\partial z_k}\right)=
  \exp\left(z_k\frac{\partial V^{(b)}}{\partial z_k}\right)=\frac{1}{w_k^j}\left(1-\frac{z_k}{z_e}\right)\left(1-\frac{z_k}{z_f}\right)^{-1}=1
\end{equation*}
is equivalent with the following completeness condition 
$$\frac{1}{w_k^j}\left(1-\frac{z_k}{z_f}\right)^{-1}\left(1-\frac{z_k}{z_e}\right)=1$$
along a meridian $m$ in Figure \ref{cases2}(b).
Here, $b_j$, $c_j$, $d_j$, $a_{j+1}$, $b_{j+1}$, $c_{j+1}$ in Figure \ref{cases2}(b) are the points of the cusp diagram, 
which lie on the edges ${\rm B}_j{\rm F}_j$, ${\rm C}_j{\rm F}_j$, ${\rm D}_j{\rm F}_j$,
${\rm A}_{j+1}{\rm E}_{j+1}$, ${\rm B}_{j+1}{\rm E}_{j+1}$, ${\rm C}_{j+1}{\rm E}_{j+1}$
of Figure \ref{twistocta}(a), respectively.
(To simplify the cusp diagram in Figure \ref{cases2}(b), we subdivided the polygon ${\rm A}_j{\rm B}_j{\rm C}_j{\rm D}_j{\rm F}_j$
in Figure \ref{twistocta}(c) into three tetrahedra by adding the edge ${\rm B}_j{\rm D}_j$.)

The part of the potential function $V$ containing $z_k$ in Figure \ref{cases}(c) is
\begin{equation*}
  V^{(c)}=-\log w_k^j\log z_k+\li\left(\frac{z_e}{z_k}\right)-\li\left(\frac{z_f}{z_k}\right),
\end{equation*}
and
\begin{equation*}
  \exp\left(z_k\frac{\partial V}{\partial z_k}\right)=
    \exp\left(z_k\frac{\partial V^{(c)}}{\partial z_k}\right)=\frac{1}{w_k^j}\left(1-\frac{z_e}{z_k}\right)\left(1-\frac{z_f}{z_k}\right)^{-1}=1
\end{equation*}
is equivalent with the following gluing equation
$${w_k^j}\left(1-\frac{z_e}{z_k}\right)^{-1}\left(1-\frac{z_f}{z_k}\right)=1$$
of $c_j=c_{j+1}$ in Figure \ref{cases2}(c).
(Compare it with Figure 12 of \cite{Cho13a}.)
Here, $b_j$, $c_j$, $d_j$, $b_{j+1}$, $c_{j+1}$, $d_{j+1}$ in Figure \ref{cases2}(c) are the points of the cusp diagram, 
which lie on the edges ${\rm B}_j{\rm F}_j$, ${\rm C}_j{\rm F}_j$, ${\rm D}_j{\rm F}_j$,
${\rm B}_{j+1}{\rm F}_{j+1}$, ${\rm C}_{j+1}{\rm F}_{j+1}$, ${\rm D}_{j+1}{\rm F}_{j+1}$
of Figure \ref{twistocta}(a), respectively,
and the edges $d_j c_j$ and $b_j c_j$ are identified to $b_{j+1}c_{j+1}$ and $d_{j+1}c_{j+1}$, respectively.
(To simplify the cusp diagram in Figure \ref{cases2}(c), we subdivided the polygon ${\rm A}_j{\rm B}_j{\rm C}_j{\rm D}_j{\rm F}_j$
in Figure \ref{twistocta}(c) into three tetrahedra by adding the edge ${\rm B}_j{\rm D}_j$.)

The part of the potential function $V$ containing $z_k$ in Figure \ref{cases}(d) is
\begin{equation*}
  V^{(d)}=\log w_k^j\log z_k-\li\left(\frac{z_k}{z_e}\right)+\li\left(\frac{z_k}{z_f}\right),
\end{equation*}
and
\begin{equation*}
  \exp\left(z_k\frac{\partial V}{\partial z_k}\right)=
  \exp\left(z_k\frac{\partial V^{(d)}}{\partial z_k}\right)={w_k^j}\left(1-\frac{z_k}{z_e}\right)\left(1-\frac{z_k}{z_f}\right)^{-1}=1
\end{equation*}
is equivalent with the following gluing equation
$${w_k^j}\left(1-\frac{z_k}{z_e}\right)\left(1-\frac{z_k}{z_f}\right)^{-1}=1$$
of $b_j=b_{j+1}$ in Figure \ref{cases2}(d).
(Compare it with Figure 13 of \cite{Cho13a}.)
Here, $a_j$, $b_j$, $c_j$, $a_{j+1}$, $b_{j+1}$, $c_{j+1}$ in Figure \ref{cases2}(d) are the points of the cusp diagram, 
which lie on the edges ${\rm A}_j{\rm E}_j$, ${\rm B}_j{\rm E}_j$, ${\rm C}_j{\rm E}_j$,
${\rm A}_{j+1}{\rm E}_{j+1}$, ${\rm B}_{j+1}{\rm E}_{j+1}$, ${\rm C}_{j+1}{\rm E}_{j+1}$
of Figure \ref{twistocta}(a), respectively,
and the edges $a_j b_j$ and $c_j b_j$ are identified to $c_{j+1}b_{j+1}$ and $a_{j+1}b_{j+1}$, respectively.

Note that the case when both of the crossings $j$ and $j+1$ in Figure \ref{cases} are degenerate can be proved by the same way.

On the other hand, it was already shown  in \cite{Cho13a} that all hyperbolicity equations are induced by these types of equations 
(see the discussion that follows Lemma 3.1 of \cite{Cho13a}), so the proof is done.

\end{proof}

In \cite{Cho13a}, we could not prove the existence of a solution of $\mathcal{H}$, in other words $\mathcal{S}\neq \emptyset$,
so we assumed it. However, the following theorem proves the existence by directly constructing one solution
from the given boundary-parabolic representation $\rho$ together with the shadow-coloring.

\begin{thm}\label{thm1}
 Consider a shadow-coloring of a link diagram $D$ induced by $\rho$ 
 and 
 the potential function $V(z_1,\ldots,z_n, w_k^j,\ldots)$ from $D$.
 For each side of $D$ with the side variable $z_k$, arc-color $a_l$ and the region-color $s$, as in Figure \ref{pic18}, we define
 \begin{equation}\label{sol1}
   z_k^{(0)}:=\frac{\det(a_l,p)}{\det(a_l,s)}.
 \end{equation}
 Also, if the positive crossing $j$ in Figure \ref{ze=zf}(a) is degenerate, then we define
  \begin{eqnarray*}
   (w_e^j)^{(0)}:=\frac{\det(s,p)}{\det(s*a_k,p)},~
   (w_f^j)^{(0)}:=\frac{\det( (s*a_l)*a_k,p)}{\det( s*a_k,p)},\label{sol2}\\
   (w_g^j)^{(0)}:=\frac{\det( (s*a_l)*a_k,p)}{\det( s*a_l,p)},~
   (w_h^j)^{(0)}:=\frac{\det(s,p)}{\det( s*a_l,p)},\nonumber
 \end{eqnarray*}
and, if the negative crossing $j$ in Figure \ref{ze=zf}(b) is degenerate, then we define
  \begin{eqnarray*}
   (w_e^j)^{(0)}:=\frac{\det( s*a_l,p)}{\det( (s*a_l)*a_k,p)},~
   (w_f^j)^{(0)}:=\frac{\det( s*a_k,p)}{\det( (s*a_l)*a_k,p)},\label{sol2}\\
   (w_g^j)^{(0)}:=\frac{\det( s*a_k,p)}{\det(s,p)},~
   (w_h^j)^{(0)}:=\frac{\det( s*a_l,p)}{\det(s,p)}.\nonumber
 \end{eqnarray*}

 Then $z_k^{(0)}\neq 0,1,\infty$, $(w_k^j)^{(0)}\neq 0,1$ for all possible $j, k$, 
 and $(z_1^{(0)},\ldots,z_n^{(0)},(w_k^j)^{(0)},\ldots)\in\mathcal{S}$.
\end{thm}

\begin{figure}[h]
\centering  \setlength{\unitlength}{0.6cm}\thicklines
\begin{picture}(6,5)  
    \put(6,4){\vector(-1,-1){4}}
    \put(1.5,2.2){$s$}
    \put(5,1.5){$s*a_l$}
    \put(6.2,4.2){$a_l$}
    \put(4,3){$z_k$}
  \end{picture}
  \caption{Region-coloring}\label{pic18}
\end{figure}
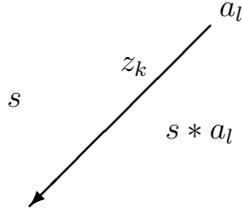

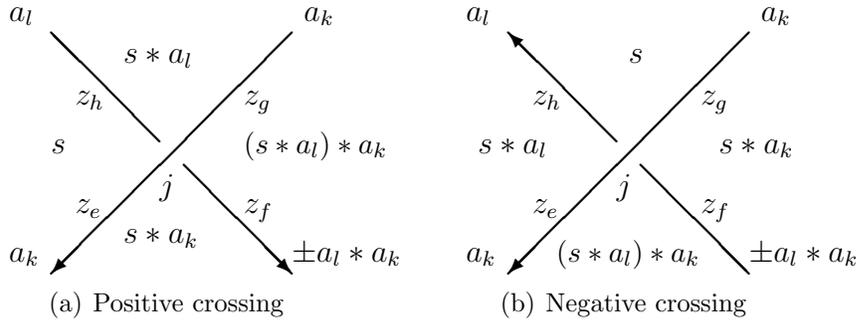
\begin{figure}[h]
\centering  \setlength{\unitlength}{0.8cm}
\subfigure[Positive crossing]{
\begin{picture}(6,5)\thicklines
    \put(5,4){\vector(-1,-1){4}}
    \put(2.8,2.2){\line(-1,1){1.8}}
    \put(3.2,1.8){\vector(1,-1){1.8}}
    \put(5.2,4.2){$a_k$}
    \put(0.3,4.2){$a_l$}
    \put(0.3,0.2){$a_k$}
    \put(5,0.2){$\pm a_l*a_k$}
    \put(2.8,1.3){$j$}
    \put(1.4,1){$z_e$}
    \put(4.2,1){$z_f$}
    \put(4.2,2.8){$z_g$}
    \put(1.4,2.8){$z_h$}  
    \put(2.2,3.5){$s*a_l$}
    \put(1,2){$s$}
    \put(2.2,0.5){$s*a_k$}
    \put(4.2,2){\small $(s*a_l)*a_k$}
\end{picture}}\hspace{1cm}
\subfigure[Negative crossing]{
\begin{picture}(6,5)\thicklines
    \put(5,4){\vector(-1,-1){4}}
    \put(2.8,2.2){\vector(-1,1){1.8}}
    \put(3.2,1.8){\line(1,-1){1.8}}
    \put(5.2,4.2){$a_k$}
    \put(0.3,4.2){$a_l$}
    \put(0.3,0.2){$a_k$}
    \put(5,0.2){$\pm a_l*a_k$}
    \put(2.8,1.3){$j$}
    \put(1.4,1){$z_e$}
    \put(4.2,1){$z_f$}
    \put(4.2,2.8){$z_g$}
    \put(1.4,2.8){$z_h$}
    \put(4.5,2){$s*a_k$}
    \put(3,3.5){$s$}
    \put(0.5,2){$s*a_l$}
    \put(1.8,0.2){\small $(s*a_l)*a_k$}
 \end{picture}}
  \caption{Crossings with shadow-colors and side-variables}\label{ze=zf}
  \end{figure}

Note that the $\pm$ signs in the arc-colors of Figure \ref{ze=zf} appears
due to the representatives of the colors in $\mathbb{C}^2\backslash\{0\}$.
However, $\pm$ does not change the value of $z_k^{(0)}$ because 
 \begin{equation*}
   \frac{\det(\pm a_l,p)}{\det(\pm a_l,s)}=\frac{\det(a_l,p)}{\det(a_l,s)}=z_k^{(0)}.
 \end{equation*}
Likewise, the value of $(w_k^j)^{(0)}$ does not depend on the choice of $\pm$ because the representatives of region-colors
are uniquely determined from the fact $s*(\pm a)=s*a$ for any $s,a\in\mathbb{C}^2\backslash\{0\}$.

\begin{proof}[Proof of Theorem \ref{thm1}]
At first, when the crossing $j$ in Figure \ref{ze=zf} is degenerate, we will show
\begin{equation}\label{zefgh}
z_e^{(0)}=z_f^{(0)}=z_g^{(0)}=z_h^{(0)},
\end{equation}
which satisfies (\ref{partialwj}). Using $h(a_k)=h(a_l)$, we put 
$a_k=\left(\begin{array}{cc}\alpha&\beta\end{array}\right)$ and
$a_l=\left(\begin{array}{cc}c\,\alpha &c\,\beta\end{array}\right)=c\,a_k$ for some constant $c\in\mathbb{C}\backslash\{0\}$.
Then we obtain $a_l*a_k=a_l$ and, if $j$ is positive crossing, then
\begin{eqnarray*}
&&z_e^{(0)}=\frac{c\,\det(a_k,p)}{c\,\det(a_k,s)}=\frac{\det(a_l,p)}{\det(a_l,s)}=z_h^{(0)},\\
&&z_f^{(0)}=\frac{\det(\pm a_l*a_k,p)}{\det(\pm a_l*a_k,s*a_k)}=\frac{\det(a_l*a_k,p)}{\det(a_l*a_k,s*a_k)}
=\frac{\det(a_l,p)}{\det(a_l,s)}=z_h^{(0)},\\
&&z_g^{(0)}=\frac{c\,\det(a_k,p)}{c\,\det(a_k,s*a_l)}=\frac{\det(a_l,p)}{\det(a_l,s*a_l)}=z_h^{(0)}.
\end{eqnarray*}
If $j$ is negative crossing, then by exchanging the indices $e\leftrightarrow g$ in the above calculation, we obtain the same result.

Note that Lemma \ref{lem1} and the definition of $p$ in Section \ref{triang} guarantee 
$z_k^{(0)}\neq 0,1,\infty$ and $(w_k^j)^{(0)}\neq 0,1$, 
so we will concentrate on proving $(z_1^{(0)},\ldots,z_n^{(0)},(w_k^j)^{(0)},\ldots)\in\mathcal{S}$.

Consider the positive crossing $j$ in Figure \ref{ze=zf}(a) and assume it is non-degenerate. 
Also consider the tetrahedra in Figures \ref{pic08}(a) and \ref{coordiocta}(a), 
and assign variables $z_e,z_f,z_g,z_h$ to sides of the link diagram as in Figure \ref{ze=zf}(a). 
Then, using (\ref{eq11}) and (\ref{sol1}), the shape parameters assigned to the horizontal edges 
${\rm A}_j{\rm B}_j$ and ${\rm D}_j{\rm A}_j$ are
\begin{eqnarray*}
1&\neq&  [h(s*a_k),h(p),h(\pm a_l*a_k),h(a_k)]\\
&=&\frac{\det(s,a_k)}{\det(s*a_k,\pm a_l*a_k)}\frac{\det(p,\pm a_l*a_k)}{\det(p,a_k)}
  =\frac{z_f^{(0)}}{z_e^{(0)}}, \\
1&\neq&  {[h(s), h(p), h ( a_k), h(a_l)]}
  =\frac{\det(s,a_l)}{\det(s,a_k)}\frac{\det(p,a_k)}{\det(p,a_l)}
  =\frac{z_e^{(0)}}{z_h^{(0)}},
\end{eqnarray*}
respectively. Likewise, the shape parameters assigned to ${\rm B}_j{\rm C}_j$ and ${\rm C}_j{\rm D}_j$ are 
$\frac{z_g^{(0)}}{z_f^{(0)}}$ and $\frac{z_h^{(0)}}{z_g^{(0)}}$ respectively.
Furthermore, for any $a,b\in\mathbb{C}^2\backslash\{0\}$, we can easily show that $h(a*b-a)=h(b)$. If
$\frac{z_g^{(0)}}{z_e^{(0)}}=\frac{\det(a_k,s)}{\det(a_k,s*a_l)}=1$, then $h(a_k)=h(s*a_l-s)=h(a_l)$, which is contradiction.
Therefore, we obtain $\frac{z_g^{(0)}}{z_e^{(0)}}\neq 1$, and $\frac{z_h^{(0)}}{z_f^{(0)}}\neq 1$ can be obtained similarly.

We can verify the same holds for non-degenerate negative crossing $j$
by the same way.

Now consider the case when the positive crossing $j$ in Figure \ref{ze=zf}(a) is degenerate. 
(See Figures \ref{twistocta}(c) and \ref{degocta}(a).)
Then, using (\ref{eq11}) and (\ref{sol2}), the shape parameters assigned to the edges 
${\rm F}_j{\rm A}_j$, ${\rm E}_j{\rm B}_j$, ${\rm F}_j{\rm C}_j$ and ${\rm E}_j{\rm D}_j$ in Figure \ref{twistocta}(c) are
\begin{eqnarray*}
[ h(a_k), h(s),h(p),h(s*a_l)][ h(a_k), h(s*a_k), h((s*a_l)*a_k), h(p)]\\
= \frac{\det(s,p)}{\det(s*a_k,p)}=(w_e^j)^{(0)},\\
  {[} h ( \pm a_l*a_k), h(p), h((s*a_l)*a_k), h(s*a_k) ]\\=\frac{\det(p,(s*a_l)*a_k)}{\det(p,s*a_k)}=(w_f^j)^{(0)},\\
{[} h(a_k), h((s*a_l)*a_k),h(p), h(s*a_k)][ h(a_k), h(s*a_l),h(s), h(p)]\\
= \frac{\det((s*a_l)*a_k,p)}{\det(s*a_l,p)}=(w_g^j)^{(0)},\\
  {[} {h}(a_l),h(p),h(s),h(s*a_l)]\\=\frac{\det(p,s)}{\det(p,s*a_l)}=(w_h^j)^{(0)},
\end{eqnarray*}
respectively. We can verify the same holds for degenerate negative crossing $j$
by the same way.

Therefore $(z_1^{(0)},\ldots,z_n^{(0)},(w_k^j)^{(0)},\ldots)$ 
satisfies the hyperbolicity equations of octahedral triangulation defined in Section \ref{triang}
and, from Proposition \ref{pro1},
we obtain $(z_1^{(0)},\ldots,z_n^{(0)},(w_k^j)^{(0)},\ldots)$ is a solution of $\mathcal{H}$. 
By the definition of $\mathcal{S}$, we obtain $(z_1^{(0)},\ldots,z_n^{(0)},(w_k^j)^{(0)},\ldots)\in\mathcal{S}$.

\end{proof}


To obtain the complex volume of $\rho$ from the potential function $V(z_1,\ldots,z_n,(w_k^j),\ldots)$, we modify it to
\begin{eqnarray}
\lefteqn{V_0(z_1,\ldots,z_n,(w_k^j),\ldots):=V(z_1,\ldots,z_n,(w_k^j),\ldots)}\label{V0} \\
&&-\sum_k\left(z_k \frac{\partial V}{\partial z_k}\right)\log z_k
-\sum_{j:\text{degenerate}\atop k}\left(w_k^j \frac{\partial V}{\partial w_k^j}\right)\log w_k^j.\nonumber
\end{eqnarray}
This modification guarantees the invariance of the value under the choice of any log-branch.
(See Lemma 2.1 of \cite{Cho13a}.) Note that $V_0(z_1^{(0)},\ldots,z_n^{(0)},(w_k^j)^{(0)},\ldots)$ means
the evaluation of the function $V_0(z_1,\ldots,z_n,(w_k^j),\ldots)$ at $(z_1^{(0)},\ldots,z_n^{(0)},(w_k^j)^{(0)},\ldots)$.

\begin{thm}\label{thm2}
Consider a hyperbolic link $L$, the shadow-coloring induced by $\rho$, 
the potential function $V(z_1,\ldots,z_n,(w_k^j),\ldots)$ 
and the solution $(z_1^{(0)},\ldots,z_n^{(0)},(w_k^j)^{(0)},\ldots)\in\mathcal{S}$ defined in Theorem \ref{thm1}. Then,
\begin{equation}\label{eq19}
  V_0(z_1^{(0)},\ldots,z_n^{(0)},(w_k^j)^{(0)},\ldots)\equiv i(\vol(\rho)+i\,\cs(\rho))~~({\rm mod}~\pi^2).
\end{equation}
\end{thm}

\begin{proof}
When the crossing $j$ is degenerate, direct calculation shows that 
the potential function $V_j$ of the crossing defined at (\ref{potential2}) satisfies
\begin{equation}\label{V0=0}
  (V_j)_0(z,z,z,z,w_1,w_2,w_3)=0,
\end{equation}
for any nonzero values of $z,w_1,w_2,w_3$.
To simplify the potential function, we rearrange the side variables $z_1, \ldots,z_n$ to 
$z_1,\ldots,z_r$, $z_{r+1}, z_{r+1}^1,z_{r+1}^2,z_{r+1}^3,\ldots$, $z_t,\ldots,z_t^3$ so that
all endpoints of sides with variables $z_1,\ldots,z_r$ are non-degenerate crossings and
the degenerate crossings induce $z_{r+1}^{(0)}=(z_{r+1}^1)^{(0)}=(z_{r+1}^2)^{(0)}=(z_{r+1}^3)^{(0)}$, 
$\ldots$, $z_t^{(0)}=\ldots=(z_t^3)^{(0)}$. (Refer (\ref{zefgh}).)
Then we define simplified potential function $\widehat{V}$ by
$$\widehat{V}(z_1,\ldots,z_t):=
  \sum_{j:\text{non-degenerate}}V_j(z_1,\ldots, z_r,z_{r+1},z_{r+1},z_{r+1},z_{r+1},\ldots,z_t,z_t,z_t,z_t).$$
Note that $\widehat{V}$ is obtained from $V$ by removing the potential functions (\ref{potential2}) of the degenerate crossings
and substituting the side variables $z_e,z_f, z_g, z_h$ around the degenerate crossing with $z_e$. From (\ref{V0=0}), we have
\begin{equation*}
  \widehat{V}_0(z_1^{(0)},\ldots,z_t^{(0)})=V_0(z_1^{(0)},\ldots,z_n^{(0)},(w_k^j)^{(0)},\ldots),
\end{equation*}
which suggests $\widehat{V}$ is just a simplification of $V$ with the same value.
Therefore, from now on, we will use only $\widehat{V}$ and 
substitute the side variables of the link diagram $z_{r+1}^1, z_{r+1}^2, z_{r+1}^3$ to $z_{r+1}$ and $z_t^1,\ldots,z_t^3$ to $z_t$, etc, except at Lemma \ref{integer} below.
Also, we remove octahedra (\ref{octa4}) or (\ref{octa4-1}) placed at all degenerate crossings
(in other words, the octahedra in Figure \ref{coordiocta})
because they do not have any effect on the complex volume. (See the comment below the proof of Theorem \ref{thmvol}.)

Now we will follow ideas of the proof of Theorem 1.2 in \cite{Cho13a}. 
However, due to the degenerate crossings, we will improve the proof to cover more general cases.
At first, we define $r_k$ by
\begin{equation}\label{r}
r_k\pi i=\left.z_k\frac{\partial \widehat{V}}{\partial z_k}\right\vert_{z_1=z_1^{(0)},\ldots,z_t=z_t^{(0)}},
\end{equation}
for $k=1,\ldots,t$, where $\vert_{z_1=z_1^{(0)},\ldots,z_t=z_t^{(0)}}$ 
means the evaluation of the equation at $(z_1^{(0)},\ldots,z_t^{(0)})$.
Unlike \cite{Cho13a}, we cannot guarantee $r_k$ is an even integer yet,
so we need the following lemma.

\begin{lem}\label{integer} For the value $z_k^{(0)}$ defined in Theorem \ref{thm1}, $(z_1^{(0)},\ldots,z_t^{(0)})$
is a solution of the following set of equations
$$\widehat{\mathcal{H}}=\left\{\left.\exp(z_k\frac{\partial \widehat{V}}{\partial z_k})=1~\right\vert~ k=1,\ldots,t\right\}.$$
\end{lem}

\begin{proof}
For a degenerate crossing $j$, from (\ref{potential2}), 
\begin{equation*}
V_j(z_k,z_k,z_k,z_k,w_e^j,w_f^j,w_g^j)
=(-\log w_e^j+\log w_f^j-\log w_g^j+\log w_h^j)\log z_k.
\end{equation*}
Therefore, using $\frac{w_f^j w_h^j}{w_e^j w_g^j}=1$, we obtain
\begin{equation*}
\exp\left(z_k\frac{\partial V_j}{\partial z_k}(z_k,z_k,z_k,z_k,w_e^j,w_f^j,w_g^j)\right)=1.
\end{equation*}
This equation implies that, if we substitute the variables $z_{r+1}^1, z_{r+1}^2, z_{r+1}^3$ to $z_{r+1}$ and $z_t^1,\ldots,z_t^3$ to $z_t$, etc, of the equations in $\mathcal{H}$,
then it becomes $\widehat{\mathcal{H}}$. 
Therefore, Theorem \ref{thm1} induces this lemma.

\end{proof}

As a corollary of Lemma \ref{integer}, now we know $r_k$ defined in (\ref{r}) is an even integer.

To avoid redundant complicate indices, we use $z_k$ instead of $z_k^{(0)}$ in this proof from now on.
Using the even integer $r_k$, we can denote $V_0(z_1,\ldots,z_t)$ by
\begin{equation}\label{v0r}
\widehat{V}_0(z_1,\ldots,z_t)=\widehat{V}(z_1,\ldots,z_t)-\sum_{k=1}^t r_k\pi i\log z_k.
\end{equation}

Now we introduce notations $\alpha_m, \beta_m, \gamma_l, \delta_j$ for the long-edge parameters defined in (\ref{gjk}).
We assign $\alpha_m$ and $\beta_m$ to non-horizontal edges as in Figure \ref{labeling}, where $m$ is over all sides
of the link diagram. 
(Recall that the edges ${\rm A}_j{\rm B}_j$, ${\rm B}_j{\rm C}_j$, ${\rm C}_j{\rm D}_j$ and ${\rm D}_j{\rm A}_j$
in Figure \ref{labeling} were named horizontal edges.)
We also assign $\gamma_l$ to horizontal edges, where $l$ is over all regions, and $\delta_j$ to the edge ${\rm E}_j{\rm F}_j$ inside the octahedron.
Although we have $\alpha_a=\alpha_c$ and $\beta_b=\beta_d$ because of the gluing, we use $\alpha_a$ for the tetrahedron
${\rm E}_j{\rm F}_j{\rm A}_j{\rm B}_j$ and ${\rm E}_j{\rm F}_j{\rm A}_j{\rm D}_j$, $\alpha_c$ for
${\rm E}_j{\rm F}_j{\rm C}_j{\rm B}_j$ and ${\rm E}_j{\rm F}_j{\rm C}_j{\rm D}_j$, $\beta_b$ for
${\rm E}_j{\rm F}_j{\rm A}_j{\rm B}_j$ and ${\rm E}_j{\rm F}_j{\rm C}_j{\rm B}_j$, $\beta_d$ for
${\rm E}_j{\rm F}_j{\rm C}_j{\rm D}_j$ and ${\rm E}_j{\rm F}_j{\rm A}_j{\rm D}_j$, respectively.
Note that the labeling is consistent even when some crossing is degenerate because, 
when the crossing $j$ in Figure \ref{labeling} is degenerate,
we obtain $z_a=z_b=z_c=z_d$ and, after removing the octahedron of the crossing,
the long-edge parameters satisfy $\alpha_a=\alpha_b=\alpha_c=\alpha_d$ and $\beta_a=\beta_b=\beta_c=\beta_d$.

\begin{figure}[h]
\centering
\begin{picture}(6,8)  
  \setlength{\unitlength}{0.8cm}\thicklines
        \put(4,5){\arc[5](1,1){360}}
    \put(6,7){\line(-1,-1){4}}
    \put(2,7){\line(1,-1){1.8}}
    \put(4.2,4.8){\line(1,-1){1.8}}
    \put(1.5,7.2){$z_d$}
    \put(6,7.2){$z_c$}
    \put(1.6,2.6){$z_a$}
    \put(6,2.6){$z_b$}
    \put(2.2,4){${\rm A}_j$}
    \put(5.3,4){${\rm B}_j$}
    \put(5.3,5.9){${\rm C}_j$}
    \put(2.2,5.9){${\rm D}_j$}
    \put(4.4,4.9){$j$}
  \end{picture}\hspace{2cm}
\includegraphics[scale=0.8]{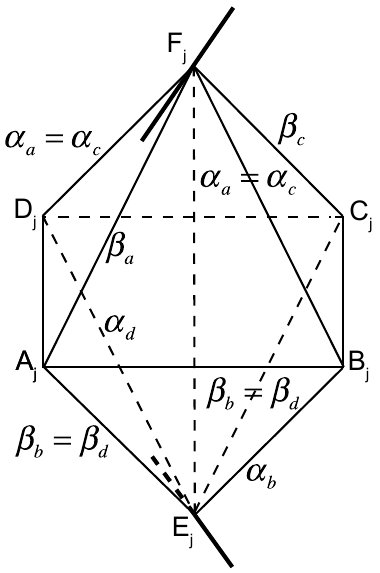}
 \caption{Long-edge parameters of non-horizontal edges}\label{labeling}
\end{figure}
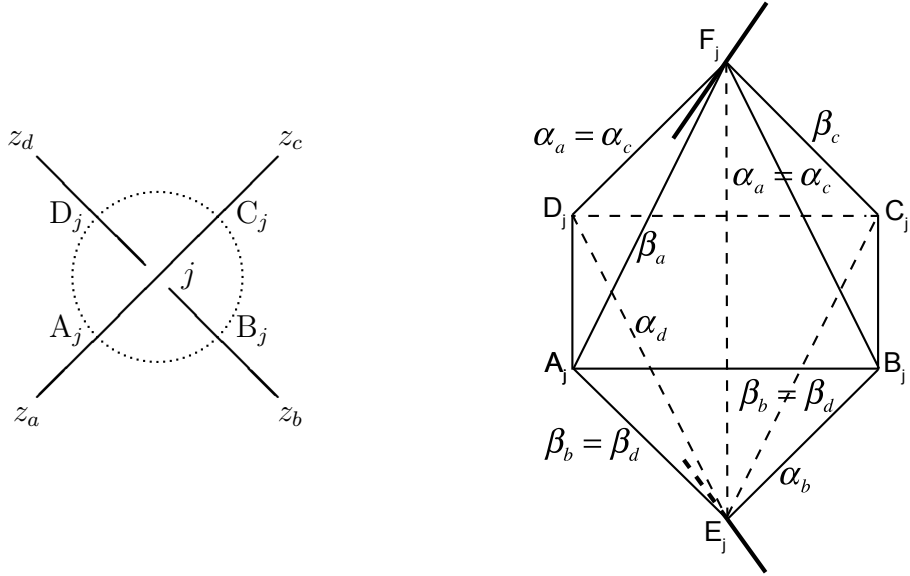

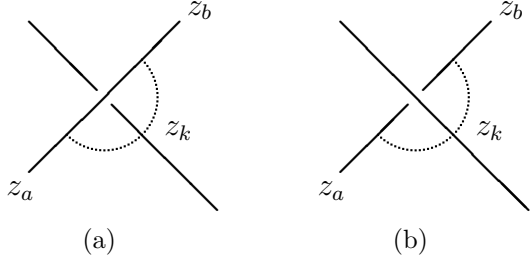
\begin{figure}[h]
\centering  \setlength{\unitlength}{1cm}\thicklines
\subfigure[]{\begin{picture}(3,3)  
  \put(1.5,1.5){\arc[5](0.5,0.5){-180}}
  \put(0.5,0.5){\line(1,1){2}}
  \put(0.5,2.5){\line(1,-1){0.9}}
  \put(1.6,1.4){\line(1,-1){1.4}}
  \put(0.2,0.2){$z_a$}
  \put(2.6,2.6){$z_b$}
  \put(2.3,1){$z_k$}
  \end{picture}}\hspace{1cm}
\subfigure[]{\begin{picture}(3,3)  
  \put(1.5,1.5){\arc[5](0.5,0.5){-180}}
  \put(0.5,0.5){\line(1,1){0.9}}
  \put(2.5,2.5){\line(-1,-1){0.9}}
  \put(0.5,2.5){\line(1,-1){2.5}}
  \put(0.2,0.2){$z_a$}
  \put(2.6,2.6){$z_b$}
  \put(2.3,1){$z_k$}
  \end{picture}}
 \caption{Two cases with respect to $z_k$}\label{twocases1}
\end{figure}

Now consider a side with variable $z_k$ and two possible cases in Figure \ref{twocases1}. We consider the case when the crossing is non-degenerate,
or equivalently, $z_a\neq z_k\neq z_b$. (If it is degenerate, we assume there is a degenerated octahedron\footnote{
Octahedron is called degenerate when two vertices at the top and the bottom coincide.} at the crossing.)
For $m=a,b$, let $\sigma_k^m\in\{\pm1\}$ be the sign of the tetrahedron\footnote{
Sign of a tetrahedron is the sign of the coordinate in (\ref{octa3}) or (\ref{octa3-1}). 
} between the sides $z_k$ and $z_m$,
and $u_k^m$ be the shape parameter of the tetrahedron assigned to the horizontal edge.
We put $\tau_k^m=1$ when $z_k$ is the numerator of $(u_k^m)^{\sigma_k^m}$ and $\tau_k^m=-1$ otherwise.
We also define $p_k^m$ and $q_k^m$ by (\ref{pq}) so that 
$\sigma_k^m[(u_k^m)^{\sigma_k^m};p_k^m,q_k^m]$ becomes the element of 
$\widehat{\mathcal{P}}(\mathbb{C})$ corresponding to the tetrahedron.
Then $\frac{1}{2}\sum_{1\leq k,m\leq t}\sigma_k^m[(u_k^m)^{\sigma_k^m};p_k^m,q_k^m]$ is the element\footnote{
The coefficient $\frac{1}{2}$ appears because the same tetrahedron is counted twice in the summation.}
 of $\widehat{\mathcal{B}}(\mathbb{C})$
corresponding to the octahedral triangulation in Section \ref{triang}, and
\begin{equation}\label{volcs}
\frac{1}{2}\sum_{1\leq k,m\leq t}\sigma_k^m \widehat{L}[(u_k^m)^{\sigma_k^m};p_k^m,q_k^m]\equiv i(\vol(\rho)+i\,\cs(\rho))\modulo,
\end{equation}
from Theorem \ref{thmvol}.

By definition, we know
\begin{equation}
u_k^a=\frac{z_k}{z_a},~u_k^b=\frac{z_b}{z_k}.\label{def_u}
\end{equation}

\begin{figure}[h]
\centering
\subfigure[]{\includegraphics[scale=0.8]{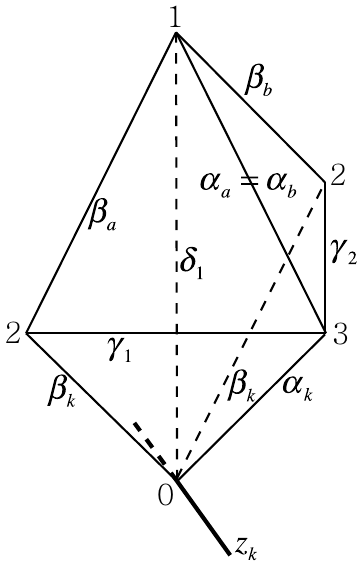}}\hspace{1cm}
\subfigure[]{\includegraphics[scale=0.8]{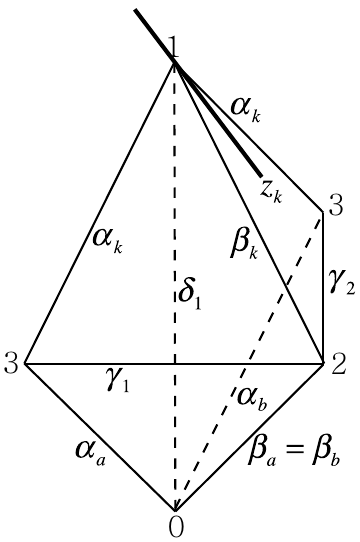}}
 \caption{Tetrahedra of Figure \ref{twocases1}}\label{twocases2}
\end{figure}

In the case of Figure \ref{twocases1}(a), we have
$$\sigma_k^a=1,~\sigma_k^b=-1~\text{ and }~\tau_k^a=\tau_k^b=1.$$
Using the equation (\ref{pq}) and Figure \ref{twocases2}(a), we decide $p_k^m$ and $q_k^m$ as follows:

\begin{eqnarray}
      \left\{\begin{array}{ll}
      \log\frac{z_k}{z_a}+p_k^a\pi i=(\log\alpha_k-\log\beta_k)-(\log\alpha_a-\log\beta_a),\\
      \log\frac{z_k}{z_b}+p_k^b\pi i=(\log\alpha_k-\log\beta_k)-(\log\alpha_b-\log\beta_b),
      \end{array}\right.\label{case a p}
\end{eqnarray}

\begin{eqnarray}
      \left\{\begin{array}{ll}
      -\log(1-\frac{z_k}{z_a})+q_k^a\pi i=\log\beta_k+\log\alpha_a-\log\gamma_1-\log\delta_1,\\
      -\log(1-\frac{z_k}{z_b})+q_k^b\pi i=\log\beta_k+\log\alpha_b-\log\gamma_2-\log\delta_1.\\
      \end{array}\right.\label{case a q}
\end{eqnarray}

In the case of Figure \ref{twocases1}(b), we have
$$\sigma_k^a=-1,~\sigma_k^b=1~\text{ and }~\tau_k^a=\tau_k^b=-1.$$
Using the equation (\ref{pq}) and Figure \ref{twocases2}(b), we decide $p_k^m$ and $q_k^m$ as follows:

\begin{eqnarray}
      \left\{\begin{array}{ll}
      \log\frac{z_a}{z_k}+p_k^a\pi i=(\log\alpha_a-\log\beta_a)-(\log\alpha_k-\log\beta_k),\\
      \log\frac{z_b}{z_k}+p_k^b\pi i=(\log\alpha_b-\log\beta_b)-(\log\alpha_k-\log\beta_k),
      \end{array}\right.\label{case b p}
\end{eqnarray}

\begin{eqnarray}
      \left\{\begin{array}{ll}
      -\log(1-\frac{z_a}{z_k})+q_k^a\pi i=\log\beta_a+\log\alpha_k-\log\gamma_1-\log\delta_1,\\
      -\log(1-\frac{z_b}{z_k})+q_k^b\pi i=\log\beta_b+\log\alpha_k-\log\gamma_2-\log\delta_1.\\
      \end{array}\right.\label{case b q}
\end{eqnarray}

The equations (\ref{case a p}) and (\ref{case b p}) holds for all (non-degenerate and degenerate) crossings, so we get the following observation.

\begin{obs}\label{obs} We have
$$\log\alpha_k-\log\beta_k\equiv\log z_k +A~~({\rm mod}~\pi i),$$
for all $k=1,\ldots,t$, where $A$ is a complex constant number independent of $k$.
\end{obs}

Note that, by the definition (\ref{potential1}), the potential function $\widehat{V}$ is expressed by
\begin{equation}\label{simpleV}
\widehat{V}(z_1,\ldots,z_t)=\frac{1}{2}\sum_{1\leq k,m\leq t}\sigma_k^m\li((u_k^m)^{\sigma_k^m})
=\frac{1}{2}\sum_{k=1}^t\sum_{m=a,\ldots,d}\sigma_k^m\li((u_k^m)^{\sigma_k^m}),
\end{equation}
where the range of the index $m$ is determined by $k$ and we put the range of $m$ by $m=a,\ldots,d$\footnote{
The range $m=a,\ldots,d$ means that each side with one of the side variables $z_a,\ldots,z_d$ share
a non-degenerate crossing with a side with $z_k$.
} from now on.
Recall that $r_k$ was defined in (\ref{r}). Direct calculation shows
\begin{equation*}
r_k\pi i=-\sum_{m=a,\ldots,d}\sigma_k^m\tau_k^m\log (1-(u_k^m)^{\sigma_k^m}).
\end{equation*}

Combining (\ref{case a q}) and (\ref{case b q}), we obtain
\begin{equation*}
\sum_{m=a,b}\sigma_k^m\tau_k^m \left\{-\log (1-(u_k^m)^{\sigma_k^m})+q_k^m\pi i\right\}
=-\log\gamma_1+\log\gamma_2,
\end{equation*}
for both cases in Figure \ref{twocases1}. 
(Note that $\alpha_a=\alpha_b$ in (\ref{case a q}) and $\beta_a=\beta_b$ in (\ref{case b q}).)
Therefore, we obtain
$$\sum_{m=a,\ldots,d}\sigma_k^m\tau_k^m \left\{-\log (1-(u_k^m)^{\sigma_k^m})+q_k^m\pi i\right\}=0,$$
and
\begin{equation}\label{rkq}
r_k\pi i =-\sum_{m=a,\ldots,d}\sigma_k^m\tau_k^m q_k^m\pi i.
\end{equation}


\begin{lem}\label{lem_a}
For all possible $k$ and $m$, we have
\begin{equation}
\frac{1}{2}\sum_{1\leq k,m\leq t}\sigma_k^m q_k^m\pi i\log(u_k^m)^{\sigma_k^m}\equiv-\sum_{k=1}^t r_k\pi i\log z_k~~({\rm mod}~2\pi^2).
\end{equation}
\end{lem}

\begin{proof} 
Note that, by definition, $\sigma_k^m=\sigma_m^k$, $\tau_k^m=-\tau_m^k$ and
$$(u_k^m)^{\sigma_k^m}=\left(\frac{z_k}{z_m}\right)^{\tau_k^m}=(z_k)^{\tau_k^m}(z_m)^{\tau_m^k}.$$
Using the above and (\ref{rkq}), we can directly calculate
\begin{eqnarray*}
\frac{1}{2}\sum_{k=1}^t\sum_{m=a,\ldots,d}\sigma_k^m q_k^m\pi i\log (u_k^m)^{\sigma_k^m}
&\equiv&\sum_{k=1}^t\left(\sum_{m=a,\ldots,d}\sigma_k^m\tau_k^m q_k^m\pi i\right)\log z_k~~({\rm mod}~2\pi^2)\\
&=&-\sum_{k=1}^t r_k\pi i\log z_k.
\end{eqnarray*}
\end{proof}

\begin{lem}\label{lem_b}
For all possible $k$ and $m$, we have
$$\frac{1}{2}\sum_{1\leq k,m\leq t}\sigma_k^m\log\left(1-(u_k^m)^{\sigma_k^m}\right)\left(\log(u_k^m)^{\sigma_k^m}+p_k^m\pi i\right)
\equiv-\sum_{k=1}^t r_k\pi i\log z_l~~({\rm mod}~2\pi^2).$$
\end{lem}

\begin{proof} 
From (\ref{case a p}) and (\ref{case b p}), we have
$$\log(u_k^m)^{\sigma_k^m}+p_k^m\pi i=\tau_k^m(\log \alpha_k-\log \beta_k)+\tau_m^k(\log \alpha_m-\log \beta_m).$$
Therefore,
\begin{eqnarray*}
\lefteqn{\frac{1}{2}\sum_{1\leq k,m\leq t}\sigma_k^m\log\left(1-(u_k^m)^{\sigma_k^m}\right)\left(\log(u_k^m)^{\sigma_k^m}+p_k^m\pi i\right)}\\
&&=\sum_{k=1}^t\left(\sum_{m=a,\ldots,d}\sigma_k^m\tau_k^m\log(1-(u_k^m)^{\sigma_k^m})\right)(\log\alpha_k-\log\beta_k)\\
&&=-\sum_{k=1}^t r_k\pi i(\log\alpha_k-\log\beta_k).
\end{eqnarray*}

Note that 
$$\sum_{k=1}^t r_{k}\pi i=\sum_{k=1}^t z_k\frac{\partial \widehat{V}}{\partial z_k}=0$$
because $ \widehat{V}$ is expressed by the summation of certain forms of $\li(\frac{z_a}{z_b})$ and
$$z_a\frac{\partial\li(z_a/z_b)}{\partial z_a}+z_b\frac{\partial\li(z_a/z_b)}{\partial z_b}=
-\log(1-\frac{z_a}{z_b})+\log(1-\frac{z_a}{z_b})=0.$$

By using Observation \ref{obs}, the above and the fact that $r_k$ is even, we have
\begin{equation*}
-\sum_{k=1}^t r_k\pi i(\log\alpha_k-\log\beta_k)\equiv-\sum_{k=1}^t r_k\pi i(\log z_k+A)
=-\sum_{k=1}^t r_k\pi i\log z_k~~({\rm mod}~2\pi^2).
\end{equation*}

\end{proof}

Combining (\ref{volcs}), (\ref{simpleV}), Lemma \ref{lem_a} and Lemma \ref{lem_b}, we complete the proof of Theorem \ref{thm2} as follows:
\begin{eqnarray*}
\lefteqn{i(\vol(\rho)+i\,\cs(\rho))
\equiv\frac{1}{2}\sum_{1\leq k,m\leq t}\sigma_k^m \widehat{L}[(u_k^m)^{\sigma_k^m};p_k^m,q_k^m]}\\
&&=\frac{1}{2}\sum_{1\leq k,m\leq t}\sigma_k^m\left(\li\left((u_k^m)^{\sigma_k^m}\right)-\frac{\pi^2}{6}\right)
+\frac{1}{4}\sum_{1\leq k,m\leq t}\sigma_k^m q_k^m\pi i\log\left(u_k^m\right)^{\sigma_k^m}\\
&&~~+\frac{1}{4}\sum_{1\leq k,m\leq t}\sigma_k^m\log\left(1-\left(u_k^m\right)^{\sigma_k^m}\right)
\left(\log\left(u_k^m\right)^{\sigma_k^m}+p_k^m\pi i\right)\\
&&\equiv \widehat{V}(z_1,\ldots,z_n)-\sum_{k=1}^t r_k\pi i \log z_k=\widehat{V}_0(z_1,\ldots,z_t)\modulo.
\end{eqnarray*}

\end{proof}

\section{Examples}\label{sec4}
\subsection{Figure-eight knot $4_1$}\label{sec41}

\begin{figure}[h]\centering
\includegraphics[scale=0.45]{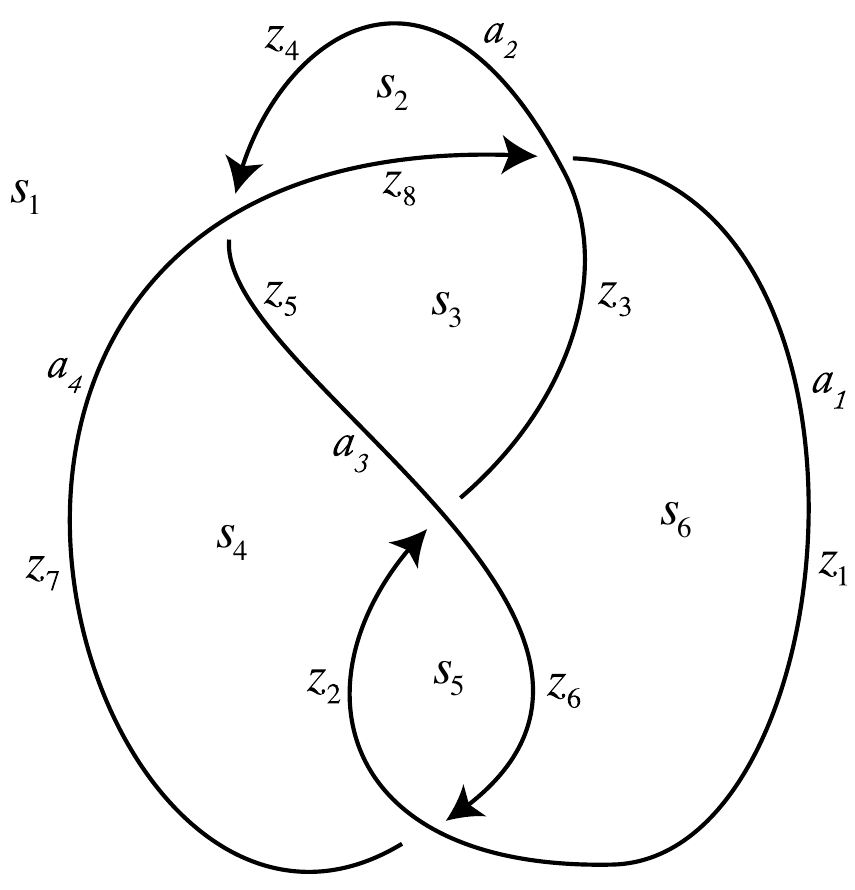}
\caption{Figure-eight knot $4_1$ with parameters}\label{example1}
\end{figure}

For the figure-eight knot diagram in Figure \ref{example1}, 
let the elements of $\mathcal{P}$ corresponding to the arcs be 
$$a_1=\left(\begin{array}{cc}0 &t\end{array}\right),~a_2=\left(\begin{array}{cc}1 &0\end{array}\right),~
a_3=\left(\begin{array}{cc}-t &1+t\end{array}\right),~a_4=\left(\begin{array}{cc}-t &t\end{array}\right),$$
where $t$ is a solution of $t^2+t+1=0$. These elements satisfy
\begin{equation}\label{minus1}
a_1*a_2=a_4,~a_3*a_4=a_2,~a_1*a_3=-a_2,~a_3*a_1=a_4,
\end{equation}
where the identities are expressed in $\mathbb{C}^2\backslash\{0\}$, not in $\mathcal{P}=(\mathbb{C}^2\backslash\{0\})/\pm$.
Let $\rho:\pi_1(4_1)\rightarrow{\rm PSL}(2,\mathbb{C})$
be the boundary-parabolic representation determined by $a_1,\ldots,a_4$.
We define the shadow-coloring of Figure \ref{example1} induced by $\rho$
by letting
\begin{eqnarray*}
s_1=\left(\begin{array}{cc}1 &1\end{array}\right),~s_2=\left(\begin{array}{cc}0 &1\end{array}\right),~
s_3=\left(\begin{array}{cc}-t-1 &t+2\end{array}\right),~s_4=\left(\begin{array}{cc}-2t-1 &2t+3\end{array}\right),\\
s_5=\left(\begin{array}{cc}-2t-1 &t+4\end{array}\right),~s_6=\left(\begin{array}{cc}1 &t+2\end{array}\right),~
p=\left(\begin{array}{cc}2 &1\end{array}\right).
\end{eqnarray*}
Direct calculation shows this shadow-coloring satisfies (\ref{regcon}) in Lemma \ref{lem1}. 
(However, this does not satisfy (\ref{exi}).) 

All values of $h(a_1),\ldots,h(a_4)$ are different, hence the potential function $V(z_1,\ldots,z_8)$ of Figure \ref{example1}
is (\ref{V41}). Applying Theorem \ref{thm1}, we obtain
\begin{eqnarray*}
z_1^{(0)}=\frac{\det(a_1,p)}{\det(a_1,s_6)}=2,~z_2^{(0)}=\frac{\det(a_1,p)}{\det(a_1,s_5)}=\frac{-2}{2t+1},
~z_3^{(0)}=\frac{\det(a_2,p)}{\det(a_2,s_6)}=\frac{1}{t+2},\\
z_4^{(0)}=\frac{\det(a_2,p)}{\det(a_2,s_1)}=1,~z_5^{(0)}=\frac{\det(a_3,p)}{\det(a_3,s_4)}=-3t-2,
~z_6^{(0)}=\frac{\det(a_3,p)}{\det(a_3,s_5)}=\frac{3t+2}{2t},\\
z_7^{(0)}=\frac{\det(a_4,p)}{\det(a_4,s_4)}=\frac{3}{2},~z_8^{(0)}=\frac{\det(a_4,p)}{\det(a_4,s_3)}=3,
\end{eqnarray*}
and $(z_1^{(0)},\ldots,z_8^{(0)})$ becomes a solution of $\mathcal{H}=\{\exp(z_k\frac{\partial V}{\partial z_k})=1~|~k=1,\ldots,8\}$.
Applying Theorem \ref{thm2}, we obtain
\begin{equation*}
V_0(z_1^{(0)},\ldots,z_8^{(0)})\equiv i(\vol(\rho)+i\,\cs(\rho))\modulo,
\end{equation*}
and numerical calculation verifies it by
\begin{equation*}
V_0(z_1^{(0)},\ldots,z_8^{(0)})=
      \left\{\begin{array}{ll}i(2.0299...+0\,i)=i(\vol(4_1)+i\,\cs(4_1))&\text{ if }t=\frac{-1-\sqrt{3} \,i}{2}, \\
                i(-2.0299...+0\,i)=i(-\vol(4_1)+i\,\cs(4_1))&\text{ if }t=\frac{-1+\sqrt{3}\,i}{2}. \end{array}\right.
\end{equation*}

\subsection{Trefoil knot $3_1$}\label{sec42}

\begin{figure}[h]\centering
\includegraphics[scale=0.5]{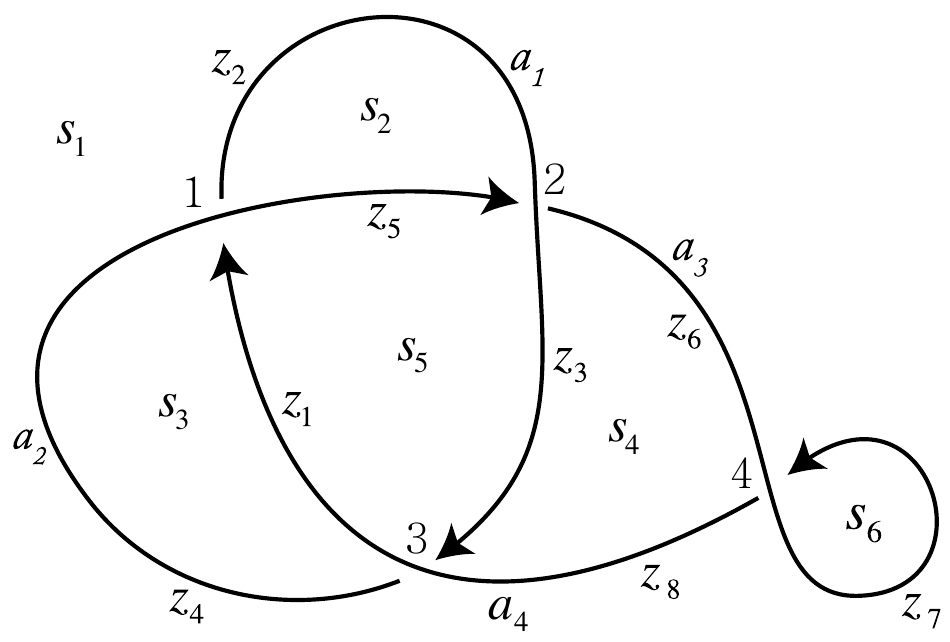}
\caption{Trefoil knot $3_1$ with parameters}\label{example2}
\end{figure}

For the trefoil knot diagram in Figure \ref{example2}, 
let the elements of $\mathcal{P}$ corresponding to the arcs be 
$$a_1=\left(\begin{array}{cc}1 &0\end{array}\right),~a_2=\left(\begin{array}{cc}0 &1\end{array}\right),~
a_3=a_4=\left(\begin{array}{cc}-1 &1\end{array}\right).$$
(Note that crossing 4 is degenerate.) These elements satisfy
\begin{equation}\label{minus2}
a_4*a_2=-a_1,~a_2*a_1=a_3,~a_1*a_4=a_2,~a_4*a_3=a_3.
\end{equation}
where the identities are expressed in $\mathbb{C}^2\backslash\{0\}$, not in $\mathcal{P}=(\mathbb{C}^2\backslash\{0\})/\pm$.
Let $\rho:\pi_1(3_1)\rightarrow{\rm PSL}(2,\mathbb{C})$
be the boundary-parabolic representation determined by $a_1,a_2,a_3, a_4$.
We define the shadow-coloring of Figure \ref{example1} induced by $\rho$
by letting
\begin{eqnarray*}
s_1=\left(\begin{array}{cc}-1 &2\end{array}\right),~s_2=\left(\begin{array}{cc}1 &2\end{array}\right),~
s_3=\left(\begin{array}{cc}-1 &3\end{array}\right),~s_4=\left(\begin{array}{cc}0 &1\end{array}\right),\\
s_5=\left(\begin{array}{cc}1 &1\end{array}\right),s_6=\left(\begin{array}{cc}-2 &3\end{array}\right),
~p=\left(\begin{array}{cc}2 &1\end{array}\right).
\end{eqnarray*}
Direct calculation shows this shadow-coloring satisfies (\ref{regcon}) in Lemma \ref{lem1}. 
(However, this does not satisfy (\ref{exi}).) 

All values of $h(a_1),h(a_2),h(a_3)=h(a_4)$ are different, 
hence the potential function 
$V$ of Figure \ref{example2} is 
\begin{eqnarray*}
V(z_1,\ldots,z_8,w_6^4,w_7^4)&=&\li(\frac{z_2}{z_5})-\li(\frac{z_2}{z_4})+\li(\frac{z_1}{z_4})-\li(\frac{z_1}{z_5})\\
&+&\li(\frac{z_6}{z_3})-\li(\frac{z_6}{z_2})+\li(\frac{z_5}{z_2})-\li(\frac{z_5}{z_3})\\
&+&\li(\frac{z_4}{z_1})-\li(\frac{z_4}{z_8})+\li(\frac{z_3}{z_8})-\li(\frac{z_3}{z_1})\\
&-&\log w_6^4\log z_6+\log{w_6^4}\log z_8,
\end{eqnarray*}
and the simplified potential function $\widehat{V}$ defined in the proof of Theorem \ref{thm2} is
\begin{eqnarray*}
\widehat{V}(z_1,\ldots,z_6)&=&\li(\frac{z_2}{z_5})-\li(\frac{z_2}{z_4})+\li(\frac{z_1}{z_4})-\li(\frac{z_1}{z_5})\\
&+&\li(\frac{z_6}{z_3})-\li(\frac{z_6}{z_2})+\li(\frac{z_5}{z_2})-\li(\frac{z_5}{z_3})\\
&+&\li(\frac{z_4}{z_1})-\li(\frac{z_4}{z_6})+\li(\frac{z_3}{z_6})-\li(\frac{z_3}{z_1}).
\end{eqnarray*}
Applying Theorem \ref{thm1}, we obtain
\begin{eqnarray*}
z_1^{(0)}=\frac{\det(a_4,p)}{\det(a_4,s_5)}=\frac{3}{2},~z_2^{(0)}=\frac{\det(a_1,p)}{\det(a_1,s_2)}=\frac{1}{2},
~z_3^{(0)}=\frac{\det(a_1,p)}{\det(a_1,s_5)}=1,\\
z_4^{(0)}=\frac{\det(a_2,p)}{\det(a_2,s_3)}=-2,,~z_5^{(0)}=\frac{\det(a_2,p)}{\det(a_2,s_5)}=2,\\
z_6^{(0)}=z_7^{(0)}=z_8^{(0)}=\frac{\det(a_3,p)}{\det(a_3,s_4)}=3,\\
(w_6^4)^{(0)}=\frac{\det(s_1,p)}{\det(s_4,p)}=\frac{5}{2},~
(w_7^4)^{(0)}=\frac{\det(s_1,p)}{\det(s_6,p)}=\frac{5}{8}.
\end{eqnarray*}
Note that $(z_1^{(0)},\ldots,z_8^{(0)},(w_6^4)^{(0)},(w_7^4)^{(0)})$ and $(z_1^{(0)},\ldots,z_6^{(0)})$ 
are solutions of 
\begin{eqnarray*}
\mathcal{H}=
\left\{\exp(z_k\frac{\partial V}{\partial z_k})=1,~\exp(w_k^j\frac{\partial V}{\partial w_k^j})=1~|~j=4, ~k=1,\ldots,8\right\}\\
\text{and }\widehat{\mathcal{H}}=
\left\{\exp(z_k\frac{\partial \widehat{V}}{\partial z_k})=1~|~k=1,\ldots,6\right\},
\end{eqnarray*}
respectively.
Applying Theorem \ref{thm2}, we obtain
\begin{equation*}
V_0(z_1^{(0)},\ldots,(w_7^4)^{(0)})\equiv\widehat{V}_0(z_1^{(0)},\ldots,z_6^{(0)})\equiv i(\vol(\rho)+i\,\cs(\rho))\modulo,
\end{equation*}
and numerical calculation verifies it by
\begin{equation*}
\widehat{V}_0(z_1^{(0)},\ldots,z_6^{(0)})=i(0+1.6449...i),
\end{equation*}
where $\vol(3_1)=0$ holds trivially and $1.6449...=\frac{\pi^2}{6}$ holds numerically.

\vspace{5mm}
\begin{ack}
  The author appreciates Yuichi Kabaya and Jun Murakami for suggesting this research and having much discussion.
  Ayumu Inoue gave wonderful lectures on his work \cite{Kabaya14} at Seoul National University and it became the framework of Section \ref{sec2} of this article. 
  Many people including Hyuk Kim, Seonhwa Kim, Roland van der Veen, Hitoshi Murakami, Satoshi Nawata,
  Stephan\'{e} Baseilhac heard my talks on the result and gave many suggestions.
  {Also, the author shows special thanks to the anonymous reviewer who suggested the revised proof of Lemma \ref{lem1}.}
  
  The author is supported by Basic Science Research Program through the National Research Foundation of Korea (NRF) funded by the Ministry of Education (NRF-2015R1C1A1A02037540).

\end{ack}

\bibliography{VolConj}
\bibliographystyle{abbrv}

{
\begin{flushleft}
Busan National University of Education\\Republic of Korea

  \vspace{0.4cm}
E-mail: dol0425@bnue.ac.kr\\
\end{flushleft}}
\end{document}